\documentclass[10pt,a4paper]{amsart}
\usepackage{amsmath, amssymb, amsfonts, amsthm, float, stmaryrd, epsfig}
\usepackage[unicode, pdfborder={0 0 0 [0 0 ]},bookmarksdepth=3,bookmarksopen=true]{hyperref}
\usepackage[latin1]{inputenc}
\usepackage[capitalize]{cleveref}
\usepackage{color}
\usepackage[ps,all,arc,rotate]{xy}
\usepackage{bbm} 
\usepackage{textgreek}

\hypersetup{
    colorlinks=true,
    linkcolor=black,
    citecolor=black,
    filecolor=black,
    urlcolor=black,
}

\numberwithin{figure}{section}
\numberwithin{table}{section}

\theoremstyle{plain}
\newtheorem{thm}{Theorem}[section]
\crefname{thm}{Theorem}{Theorems}
\newtheorem*{prop*}{Proposition}
\newtheorem*{thm*}{Theorem}
\newtheorem{prop}[thm]{Proposition}
\crefname{prop}{Proposition}{Propositions}
\newtheorem{lem}[thm]{Lemma}
\crefname{lem}{Lemma}{Lemmata}
\newtheorem{cor}[thm]{Corollary}
\crefname{cor}{Corollary}{Corollaries}

\newtheorem{conj}[thm]{Conjecture}
\crefname{conj}{Conjecture}{Conjectures}
\crefname{equation}{Equation}{Equations}

\theoremstyle{definition}
\newtheorem{ex}[thm]{Example}

\newtheorem{dfn}[thm]{Definition}
\newtheorem*{dfn*}{Definition}

\theoremstyle{remark}
\newtheorem{rmk}[thm]{Remark}

\newtheoremstyle{maintheorem}{}{}{\itshape}{}{\bfseries}{}{.5em}{#1 \!\thmnote{#3}.}
\theoremstyle{maintheorem}

\makeatletter
\let\c@figure\c@thm
\let\c@table\c@thm
\makeatother
\crefname{figure}{Figure}{Figures}
\crefname{table}{Table}{Tables}

\newcommand{\step}[1]{\smallskip \noindent \textbf{Step #1:}}

\newcommand{\Aut}{\operatorname{Aut}}

\newcommand{\GL}{\operatorname{GL}}

\newcommand{\id}{\operatorname{id}}
\newcommand{\im}{\operatorname{im}}
\newcommand{\supp}{\operatorname{supp}}

\newcommand{\I}{\mathrm{I}}

\newcommand{\fab}[1]{#1^\mathrm{fab}}

\renewcommand{\det}{\operatorname{det}\nolimits}

\newcommand{\typeF}[1]{\mathtt{F}_{#1}}
\newcommand{\typeFP}[1]{\mathtt{FP}_{#1}}

\def\C{\mathbb{C}}
\def\R{\mathbb{R}}
\def\N{\mathbb{N}}
\def\Z{\mathbb{Z}}
\def\Q{\mathbb{Q}}

\def\1{\mathbbm{1}}

\def\K{\mathbb{K}}

\def\LL{\mathbb{L}}

\renewcommand{\S}{\mathcal{S}}

\renewcommand{\P}{\mathcal P}

\newcommand{\D}{\mathcal{D}}

\def\s-{\smallsetminus}
\def\into{\hookrightarrow}

\def\iff{if and only if }

\newcounter{dawidcomments}

\author{Dawid Kielak}
\title[The BNS invariants via Newton polytopes]{The Bieri--Neumann--Strebel invariants via Newton polytopes}
\date{\today}

\begin{document}

\begin{abstract}
We study the Newton polytopes of determinants of square matrices defined over rings of twisted Laurent polynomials. We prove that such Newton polytopes are single polytopes (rather than formal differences of two polytopes); this result can be seen as analogous to the fact that determinants of matrices over commutative Laurent polynomial rings are themselves polynomials, rather than rational functions. We also exhibit a relationship between the Newton polytopes and invertibility of the matrices over Novikov rings, thus establishing a connection with the invariants of Bieri--Neumann--Strebel (BNS) via a theorem of Sikorav.

We offer several applications: we reprove Thurston's theorem on the existence of a polytope controlling the BNS invariants of a $3$-manifold group; we extend this result to free-by-cyclic groups, and the more general descending HNN extensions of free groups. We also show that the BNS invariants of Poincar\'e duality groups of type $\typeF{}$ in dimension $3$ and groups of deficiency one  are determined by a polytope, when the groups are assumed to be agrarian, that is their integral group rings embed in skew-fields. The latter result partially confirms a conjecture of Friedl.

We also deduce the vanishing of the Newton polytopes associated to elements of the Whitehead groups of many groups satisfying the Atiyah conjecture. We use this to show that the $L^2$-torsion polytope of Friedl--L\"uck is invariant under homotopy. We prove the vanishing of this polytope in the presence of amenability, thus proving a conjecture of Friedl--L\"uck--Tillmann.
\end{abstract}

\maketitle
\section{Introduction}

In 1986, W. Thurston \cite{Thurston1986} proved a remarkable theorem about the ways in which a compact oriented $3$-manifold $M$ can fibre over the circle. He introduced a semi-norm $x \colon H^1(M;\R) \to [0, \infty)$ (now called the Thurston norm and often denoted by $\| \cdot \|_T$), whose unit ball is a polytope $B_x$, such that if $\phi \colon \pi_1(M) \to \Z$ is a homomorphism induced by a fibration of $M$, then $\phi$ lies in the cone of an open maximal face of $B_x$, and every other integral character $\pi_1(M) \to \Z$ lying in the same cone also comes from a fibration (hence one can talk about fibred faces of $B_x$).

A year later, Bieri--Neumann--Strebel \cite{Bierietal1987} introduced the geometric (or $\Sigma$, or BNS) invariant of finitely generated groups. They observed that if $G$ is the fundamental group of the manifold $M$ above, then the $\Sigma$-invariant $\Sigma(G) \subseteq H^1(G;\R) \s- \{0\}$ coincides with the union of the cones of the fibred faces (with the origin removed). Thus, Thurston's theorem can be reinterpreted as a result about the structure of $\Sigma(G)$ -- for $3$-manifold groups, the BNS invariant is determined by the integral polytope $B_{x^\ast} \subset H_1(G;\R)$ (which is the dual polytope of $B_x$), where `integral' means that the vertices lie on the integer lattice, and `determined' means that if a character lies in $\Sigma(G)$, then any other character which attains its minimum on $B_{x^\ast}$ at the same face as $\phi$ also lies in $\Sigma(G)$. We can thus mark the vertices of $B_{x^\ast}$ dual to the fibred faces, and say that $\Sigma(G)$ is determined by a marked polytope.

The BNS invariants were subsequently computed for many finitely generated groups: finitely generated nilpotent or  metabelian groups \cite[Theorem E]{BieriGroves1984}, $2$-generator $1$-relator groups and Houghton groups \cite[Theorem 4.2 and Proposition 8.3]{Brown1987}, right-angled Artin groups (RAAGs) \cite[Theorem 5.1]{MeierVanWyk1995},
groups with staggered presentations and graph products (with underlying graphs without a central vertex)
\cite{GarilleMeier1998},
pure symmetric automorphism groups of finitely generated free groups \cite{Orlandi-Korner2000}, limit groups \cite[Corollary 30]{Kochloukova2010}, fundamental groups of compact K\"ahler manifolds \cite{Delzant2010}, pure symmetric automorphism of RAAGs \cite[Theorem A]{KobanPiggott2014},
braided Thompson group $F$
\cite[Theorem 3.4]{Zaremsky2018},
some Artin groups \cite{AlmeidaKochloukova2015a,AlmeidaKochloukova2015,Almeida2017}, pure braid groups \cite{Kobanetal2015}, \{finitely generated free\}-by-$\Z$ groups with polynomially growing monodromy \cite{CashenLevitt2016}, and Lodha--Moore groups \cite{Zaremsky2016}.
In all of these examples, $\Sigma(G)$ is determined by an integral polytope.

The invariant has also been computed in \cite[Theorem 8.1]{Bierietal1987} for finitely generated groups of piecewise linear automorphisms of the interval $[0,1]$, which are \emph{irreducible} in the sense that there are no points in the interior of the interval fixed by the entire group, and which have \emph{independent left and right derivatives} in the following sense: let $\lambda$ and $\rho$ be the left and right derivatives taken at the endpoints of the interval $[0,1]$; `independence' means that $\lambda(\ker \rho) = \lambda(G)$ and vice versa. In this case, $\Sigma(G) = H^1(G;\R) \s- \{\kappa \log \lambda, \kappa \log \rho \mid \kappa \geqslant 0 \}$. Thus, $\Sigma(G)$ is determined by a polytope, but not necessarily integral -- there is a specific example giving a non-integral polytope. The example is shown by Stein \cite{Stein1992} to be finitely presented and of type $\typeFP{\infty}$, but it is not of type $\typeF{}$ -- this is relatively easy to see, since the example contains Thompson's group $F$.

It is worth noting that the vast majority of the papers above computes $\Sigma(G)$, and only then deduces that it is determined by a polytope. It is of course very valuable to have a complete description of $\Sigma(G)$, but such detailed knowledge cannot be obtained even for all finitely presented groups -- it was shown by Cavallo--Delgado--Kahrobaei--Ventura~\cite[Theorem 6.4]{Cavalloetal2017} that computing $\Sigma(G)$ is undecidable. Thus, it is worthwhile to study the structure of $\Sigma(G)$ from a more abstract viewpoint.

The original motivation behind this article was to study the BNS invariants of free-by-cyclic groups. In particular, the author set out to prove that such BNS invariants have only finitely many connected components, which was not known prior to this work, but expected due to the analogies between free-by-cyclic groups and fundamental groups of $3$-manifolds. It turned out that the analogy goes much further, and the methods developed found also other applications. Let us summarise the results in this direction by the following statement.

\begin{thm*}
 Let $G$ belong to one of the following classes:
 \begin{enumerate}
  \item descending HNN extensions of finitely generated free groups; or
  \item fundamental groups of compact connected oriented $3$-manifolds; or
   \item agrarian Poincar\'e duality groups of dimension $3$ and type $\typeF{}$; or
  \item agrarian groups of deficiency one.
 \end{enumerate}
Then $\Sigma(G)$ is determined by an integral polytope.
\end{thm*}

Here, `agrarian' means that the integral group ring of the group in question embeds into a skew-field (division algebra).

\smallskip
We can also draw a number of conclusions related to the $L^2$-torsion polytope $P_{L^2}$ of Friedl--L\"uck~\cite{FriedlLueck2017}. This polytope is associated to every finite $L^2$-acyclic (i.e. with vanishing $L^2$-homology) classifying space of a  group $G$ satisfying the Atiyah conjecture -- the Atiyah conjecture is a statement about integrality of the $L^2$-Betti numbers of CW-complexes on which $G$ acts freely, properly and cocompactly. (In fact, the polytope $P_{L^2}$ can also be defined for other CW-complexes, not only classifying spaces.)

The original definition of $P_{L^2}$ takes as input a classifying space rather then its fundamental group, and if two different classifying spaces are taken, then the resulting polytopes differ by the Newton polytope of some element of the Whitehead group of $G$. The Whitehead group is a ($K$-theoretic) group consisting of obstructions for $G$-homotopy equivalences of $G$-CW complexes being $G$-simple homotopy equivalences.

\begin{thm*}
Let $G$ be a finitely-generated torsion-free group satisfying the Atiyah conjecture. The Newton polytope of any element of the Whitehead group $\mathrm{Wh}(G)$ vanishes. Hence, if $G$ is additionally $L^2$-acyclic and of type $\typeF{}$, then the $L^2$-torsion polytope $P_{L^2}(G)$ of $G$ is well-defined.

Also, if $G$ is amenable and $G \not \cong \Z$, then $P_{L^2}(G)$ is a singleton.
\end{thm*}

\smallskip
The methods used in the proofs of both theorems above are of an algebraic character: we use Sikorav's theorem (\cref{sikorav}) to relate BNS invariants to matrices over the group ring, and study these by extending scalars to a skew-field (given by the property of being agrarian -- usually, we will use the skew-field constructed by Linnell~\cite{Linnell1993} for torsion-free groups satisfying the Atiyah conjecture). Thus, the methods can be traced to the theory of $L^2$-invariants (already used for the computation of the BNS invariants of some $2$-generator $1$-relator groups by Friedl--Tillmann~\cite{FriedlTillmann2015}), but stripped  of its analytic flavour.
Let us also mention the article of Friedl--L\"uck~\cite{FriedlLueck2017}, which is foundational for the point of view presented here.

The paper is divided into four parts. Let us briefly discuss the contents of each of these.

\subsection*{Twisted group rings, biorderable groups, and the Ore localisation}

In \cref{sec prelims} we recall the three notions, which will be of central importance throughout the paper.

\subsection*{Matrices and their polytopes}
In \cref{sec dets} we prove the two main technical results of the paper. We let $H$ be a finitely generated free-abelian group, $\K$ be a skew-field, and denote by $\K H$ some twisted group ring of $H$ (this is a variation on the usual group ring, which appears naturally when one considers extensions of groups).
The ring $\K H$ can be thought of as a group ring, but also as a ring of twisted Laurent polynomials in several commuting variables over a skew-field.

We will embed $\K H$ into its skew-field of fractions $\D$ (formally, the Ore localisation), and study square matrices over $\K H$ which become invertible over $\D$. For any such matrix $A$, we have the Dieudonn\'e determinant $\det A$ at our disposal; since $\D$ is the skew-field of fractions, we can write $\det A$ as a fraction of two elements in $\K H$. Thinking of these elements as Laurent polynomials, we can take their Newton polytopes (convex hulls of their supports), and thus associate to $\det A$ a formal difference of two polytopes (one for the numerator, and one for the denominator).

The first result of significance is \cref{single poly}, which shows that in fact we can represent $\det A$ by a single polytope, denoted by $P(A)$. This rather innocuous statement has interesting consequences (listed below).
\cref{single poly} is also interesting in its own right: were our coefficients $\K$ forming a commutative field, and were the group ring $\K H$ not twisted, we would be dealing with matrices over an abelian ring of classical interest, the ring of Laurent polynomials. It is clear that in this case the determinant $\det A$ is a polynomial, and thus its Newton polytope is a single polytope. In our setting, the Dieudonn\'e determinant is not a polynomial (it is a rational function), but nevertheless we do obtain a single polytope, in analogy with the classical case.

\smallskip
The second theorem we prove in \cref{sec dets} is
\cref{K-bns matrix}.
It says that the shape of the Newton polytope $P(A)$ determines the existence of right-inverses of $A$ over the Novikov rings $\widehat {\K H}^\phi$. The superscript $\phi$ denotes a character $\phi \colon H \to \R$, and the Novikov ring $\widehat {\K H}^\phi$ is the ring of formal sums of elements in $H$ with coefficients in $\K$ (just like $\K H$), with the caveat that sums with infinite support are allowed, provided that the supports go to infinity `only in the direction of $\phi$'. \cref{K-bns matrix} says that $A$ admits a right-inverse over $\widehat {\K H}^\phi$ \iff $\phi$ attains its minimum on $P(A)$ at a unique vertex.

The Novikov rings play a central role in the $\Sigma$-theory (theory of the BNS invariants), since Sikorav~\cite{Sikorav1987} proved that a character $\phi$ lies in $\Sigma(G)$ \iff $H_1(G;\widehat{ \Z G}^\phi)=0$ (see \cref{sikorav}). Motivated by this result, we introduce the notion of the $\Sigma$-invariant of the matrix $A$, denoted by $\Sigma(A)$; \cref{K-bns matrix} tells us that $\Sigma(A)$ is determined by $P(A)$, and this will be the source of our polytopes throughout the paper.

\subsection*{Agrarian groups}
The focal point of the remainder of the paper will be the class of \emph{agrarian groups}, that is groups $G$ whose group ring $\Z G$ embeds in a skew-field. We introduce the class in \cref{sec agrarian}, where we also introduce the potentially smaller class of equivariantly agrarian groups.

It is clear that agrarian groups are torsion-free and satisfy the zero-divisor conjecture of Kaplansky; there are currently no other obstructions to being agrarian known -- there is not a single torsion-free example of a non-agrarian group. On the other hand, there are many positive examples: torsion-free amenable groups whose integral group rings have no zero-divisors, biorderable groups, and torsion-free groups satisfying the Atiyah conjecture are all examples of agrarian groups. The class is also closed under subgroups, directed unions, and many extensions, and being agrarian is a fully-residual property. Noteworthy from our point of view is that the class is known to contain all descending HNN extensions of free groups, and the fundamental groups of $3$-manifolds which fibre over the circle.

Note that \cref{sec agrarian,sec dets} are completely independent.

\subsection*{Applications}

In \cref{sec apps} we list the applications of \cref{single poly,K-bns matrix}.

Funke~\cite{Funke2018} introduced the notion of a group of \emph{polytope class}. He defined it as a subclass of the class of torsion-free groups satisfying the Atiyah conjecture and with finitely generated abelianisations. We show (in \cref{polytope class}) that in fact every group in this (a priori) larger class is of polytope class. This allows us to use the work of Funke in greater generality: we prove that the Newton polytopes of elements of the Whitehead group $\mathrm{Wh}(G)$ of a finitely-generated torsion-free group satisfying the Atiyah conjecture are trivial (\cref{Whitehead}). The Whitehead group is a quotient of $K_1(\Z G)$, the first $K$-group, by the subgroup $\{ \pm g \mid g \in G \}$. It plays an important role in homotopy theory:  $\mathrm{Wh}(G)$ is trivial \iff every homotopy equivalence between two CW-complexes with fundamental group $G$ is a simple-homotopy equivalence.
The elements of $\mathrm{Wh}(G)$ can be thought of as matrices over $\Z G$, and so it makes sense to talk about Newton polytopes of elements of $\mathrm{Wh}(G)$.
Our vanishing result can be seen as evidence towards the conjecture on triviality of the Whitehead groups of torsion-free groups.

The vanishing of the Newton polytopes of the elements of $\mathrm{Wh}(G)$ implies that the $L^2$-torsion polytope of Friedl--L\"uck \cite{FriedlLueck2017} is a homotopy invariant, and hence it can be defined as  a group invariant of $L^2$-acyclic groups of type $\typeF{}$ satisfying the Atiyah conjecture. The $L^2$-torsion polytope is obtained as the Newton polytope of (the determinant of) the universal $L^2$-torsion. This latter invariant can be thought of as an $L^2$-version of the Whitehead torsion -- we refer the reader to \cite{FriedlLueck2017}, where it is defined and discussed in more detail.

The next situation we look at occurs when the group $G$ admits a finite subnormal chain terminating in an amenable group $N$. We prove that if $N$ is not abelian, or $G$ has trivial centre, then the $L^2$-torsion polytope $P_{L^2}(G)$ is a singleton. We use this to
prove a conjecture of Friedl--L\"uck--Tillmann~\cite{Friedletal2016}, which states that if $G$ as above is amenable but not cyclic, then $P_{L^2}(G)$ is a singleton.

Afterwards, we look at an agrarian group $G$ of deficiency one. For such a group we show (\cref{defic 1}) that $\Sigma(G)$ is determined by an integral polytope. If $G$ satisfies the Atiyah conjecture then we obtain a slightly stronger result (\cref{defic 1 atiyah}). These results confirm a conjecture of Friedl (under the additional assumption of the group being agrarian). They are also interesting in view of a conjecture of Bieri, which states that any group of deficiency one with non-trivial $\Sigma(G)$ is a descending HNN extension of a free group. The BNS-invariants of descending HNN extensions will be discussed below, but the structure we exhibit for the BNS-invariants for agrarian groups of deficiency one is of the same kind as in the case of descending HNN extensions.

 We proceed to discuss precisely the descending HNN extensions of finitely generated free groups. Note that this class includes \{finitely generated free\}-by-$\Z$ groups (usually referred to as free-by-cyclic groups). It was the author's original motivation to prove that for such a group $G$, the invariant $\Sigma(G)$ has finitely many connected components -- this was not known before, except for free-by-cyclic groups with polynomially growing monodromy, see the work of Cashen--Levitt~\cite{CashenLevitt2016}.
 In fact we show more; we prove a result fully analogous to the statement of Thurston for $3$-manifolds: the $L^2$-torsion polytope $P_{L^2}(G)$ admits a marking of its vertices, and a character $\phi$ lies in $\Sigma(G)$ \iff it attains its minimum on $P_{L^2}(G)$ uniquely at a marked vertex (\cref{free by cyclic}). Note that the polytope $P_{L^2}(G)$ is (for descending HNN extensions of free groups) very much analogous to the Thurston polytope $B_{x^\ast}$ -- when the group is a $3$-manifold group (this happens for free-by-cyclic groups with geometric monodromy), the $L^2$-torsion polytope coincides with $B_{x^\ast}$ (see \cite{FriedlLueck2017}). Further analogies between $P_{L^2}(G)$ and $B_{x^\ast}$ for general descending HNN extensions of free groups were studied by Funke and the author in ~\cite{FunkeKielak2018}. In particular, $B_{x^\ast}$ contains enough information to recover the Thurston norm itself. Using the same technique one can use $P_{L^2}(G)$ to induce a function $H^1(G;\R) \to \R$. It was shown in~\cite{FunkeKielak2018} that in fact this function is also a semi-norm, and given an integral character it returns (up to a sign) the $L^2$-Euler characteristic of the kernel of the character. In particular, when the kernel is finitely generated, it returns its (usual) Euler characteristic, and so determines the rank of such a kernel (which is necessarily a free group).

 In the two final sections, we look at Poincar\'e duality groups of type $\typeF{}$ in dimension $3$, and then at $3$-manifolds themselves. In the first situation, again we prove  that $\Sigma(G)$ is determined by an integral marked polytope (\cref{pd}). When $M$ is a $3$-manifold, we reprove Thurston's theorem. Note that the proof we give is completely independent, and of an algebraic character.

\subsection*{Acknowledgements}
The author would like to thank Kai-Uwe Bux for comments on an earlier version of the article, as well as Stefan Friedl,
Fabian Henneke, Wolfgang L\"uck,
and Stefan Witzel for helpful discussions.

The author was supported by the Priority Programme 2026 \href{https://www.spp2026.de/}{`Geometry at infinity'} of the German Science Foundation (DFG).

\section{Twisted group rings, biorderable groups, and the Ore localisation}
\label{sec prelims}

We start by introducing three concepts that will be of the utmost importance.

\subsection{Twisted group rings}

Throughout the article, rings are unital and not necessarily commutative, modules are by default right-modules, and groups are discrete.

\begin{dfn}
Let $R$ be a ring. We denote its group of units by $R^\times$. An element $r$ is a \emph{zero-divisor} \iff $r$ is non-zero, and there exist $x \in R$ such that $xr = 0$ or $rx = 0$.
\end{dfn}

Let $R$ be a ring  and $G$ a group. We denote the set of functions $G \to R$ by $R^G$. Clearly, pointwise addition turns $R^G$ into an abelian group.
Since $R$ has a distinguished $0$, we can talk about supports.
\begin{dfn}[Support]
 For $x \in R^G$, we define its \emph{support} to be \[\supp x = \{ g \in G \mid x(g) \neq 0 \}\]
\end{dfn}
Note that for an element $x \in R^G$ we will use both the function notation, and the sum notation
\[
x = \sum_{g \in G} x(g) g
\]
In the sum notation, we will often ignore the elements $g$ with $x(g) = 0$, that is we will focus on the values $x$ attains on its support. The sum notation is particularly common when talking about the subgroup $R G$ of $R^G$ consisting of functions with finite support. We now explain how to endow $R G$ with a ring structure.

\begin{dfn}[Twisted group ring]
Let $\phi \colon G \to \Aut(R)$ and $\mu \colon G \times G \to R^\times$ be two functions satisfying
\begin{align*}
 \phi(g) \circ \phi(g') &= c\big(\mu(g,g')\big) \circ \phi(gg') \\
 \mu(g,g') \cdot \mu(gg',g'') &= \phi(g)\big(\mu(g',g'')\big) \cdot \mu(g,g'g'')
\end{align*}
where $c \colon R^\times \to \Aut(R)$ takes $r$ to the left conjugation by $r$. The functions $\phi$ and $\mu$ are called \emph{structure functions}, and turn $RG$ into a \emph{twisted group ring} by setting
\begin{equation}
\tag{$\star$}
\label{twisted conv}
r g \cdot r' g' = r \phi(g)(r') \mu(g,g') gg'
\end{equation}
and extending to sums by linearity (here we are using the sum notation).

When $\phi$ and $\mu$ are trivial, we say that $RG$ is \emph{untwisted} (in this case we obtain the usual group ring).
We adopt the convention that the group ring with $\Z$-coefficients is always untwisted.
\end{dfn}

The following is the key example of a twisted group ring, and we will use it repeatedly throughout the article.

\begin{ex}
\label{key example}
 Let $\Z G$ be the (untwisted) integral group ring. Let $\alpha \colon G \to H$ be a quotient with kernel $K$, and let $s \colon H \to G$ be a set-theoretic section of $\alpha$. Then the map
 \[
  g \mapsto \Big(g \cdot  \big(s \circ \alpha(g)\big)^{-1}\Big) \alpha(g)
 \]
induces an isomorphism of twisted group rings $\Z G \cong (\Z K)H$, where $\Z G$ and $\Z K$ are untwisted, and the structure functions of $(\Z K)H$ are as follows: $\phi(h)$ is the automorphism of $\Z K$ induced by the (left) conjugation by $s(h)$, and $\mu(h,h') = s(h)s(h') \big( s(hh')\big)^{-1}$.

Note that the inverse isomorphism $(\Z K)H \cong \Z G$ is given by
\[
\sum_{h \in H} x(h) h \mapsto \sum_{h \in H} x(h) s(h)
\]
where the coefficients $x(h)$ lie in $\Z K$, the sum on the left-hand side is formal, and the sum on the right-hand side is taken in $\Z G$.
\end{ex}

Note that in the generality that we introduced them, the rings $R G$ are sometimes called \emph{crossed products}.

The structure functions do not appear in the notation, but they are (implicitly) assumed to be specified. Note that the structure functions can be used to define products on other subsets of $R^G$ as follows.
Let $x,y \in R^G$ be two functions. For every $g \in G$ define $X_g$ and $Y_g$ to be the smallest subsets of $G$ such that if $g = a \cdot b$ and $(a,b) \in \supp x\times \supp y$, then $(a,b) \in X_g \times Y_g$. If $X_g$ and $Y_g$ are finite for every $g$, then \eqref{twisted conv} gives a method of calculating the product $x\cdot y$ by setting its value at $g$ to the value at $g$ of
\[
\Big( \sum_{a \in X_g} x(a) a \Big) \cdot \Big( \sum_{b \in Y_g} y(b) b \Big)
\]
We will refer to this product as induced by the \emph{twisted convolution \eqref{twisted conv}}.


\subsection{Biorderable groups}

Biorderable groups form a class of groups whose (twisted) group rings are well-behaved.

\begin{dfn}
\label{biord}
 A group $G$ is \emph{biorderable} \iff there exists a total ordering $\leqslant$ on $G$ which is invariant under right and left multiplication.
\end{dfn}

The class of biorderable groups is clearly closed under taking subgroups and arbitrary products -- for the latter, it is enough to order the factors of a product in some way, and then use lexicographic order on the resulting group. Together these two properties imply that residually biorderable groups are themselves biorderable.
The class is also closed under central extensions, and contains finitely generated free-abelian groups. Putting these properties together yields that residually torsion-free nilpotent groups are biorderable: this is a rather large class, containing free groups and surface groups (see e.g. \cite{Baumslag2010}), as well as RAAGs (as shown by Droms \cite{Droms1983}).

The class of biorderable groups is also closed under directed unions and free products -- see the book by Deroin--Navas--Rivas \cite[Section 1.1.1 and Theorem 2.1.9]{Deroinetal2016}.

The key property for us is that twisted group rings of biorderable groups with skew-field coefficients embed in skew-fields.

\begin{thm}[Malcev~{\cite{Malcev1948}}; Neumann~{\cite{Neumann1949}}]
\label{malcev-neumann}
 Let $G$ be a biorderable group, let $\K$ be a skew-field, and let $\K G$ be a twisted group ring. The ring $\K G$ embeds into a skew field $\mathcal F$, its \emph{Malcev--Neumann completion}.
\end{thm}

Let us briefly describe the construction: fix a biordering $\leqslant$ on $G$. We define $\mathcal F \subseteq \K^G$ to be the subset of those functions $G \to \K$ whose support is \emph{well-ordered} with respect to $\leqslant$, that is such that any non-empty subset of the support has a $\leqslant$-minimum. It is clear that $\mathcal F$ is an abelian group; it can be turned into a ring using the twisted convolution \eqref{twisted conv} -- one has to argue that when taking products, it is always sufficient to take product of two finitely supported functions, but this follows from the fact that functions in $\mathcal F$ have well-ordered supports. It turns out that the ring $\mathcal F$ is actually a skew-field.

\begin{lem}
 \label{units of biord}
 Let $G$ be a biorderable group, let $\K$ be a skew-field, and let $\K G$ denote some twisted group ring of $G$. The units of $\K G$ have supports of cardinality $1$.
\end{lem}
\begin{proof}
 Let $\leqslant$ be a biordering of $G$. Take $x \in \K G^\times$; by definition, there exists $y \in \K G$ such that $xy = 1$. Write
 \[
  x = \sum_{i=0}^n \lambda_i g_i, \ y = \sum_{j=o}^m \nu_j h_j
 \]
where the finite sequences $(g_i)$ and $(h_j)$ of elements of $G$ are strictly $\leqslant$-increasing, and the coefficients $\lambda_i$ and $\nu_j$ are never $0$.

We are now going to carry out the twisted multiplication and compute $x\cdot y$. It is immediate that the support of $xy$ does not contain any group element $\leqslant$-smaller than $g_0 h_0$. The coefficient at $g_0h_0$ of $xy$ is
\[\
 \lambda_0 \phi(g_0)(\nu_0)\mu(g_0,h_0)
\]
Crucially, none of the factors above are zero, and the multiplication is carried out in a skew-field $\K$. Therefore $g_0h_0$ is the strictly $\leqslant$-smallest element of the support of $xy$.

Arguing analogously, we conclude that $g_nh_m$ is the strictly $\leqslant$-greatest element of the support of $xy$. But $xy = 1$, and therefore $1$ is the only element in $\supp xy$. We conclude that $g_0 h_0 = g_n h_m$, which is only possible when $n=0=m$, in which case $x$ has support of cardinality $1$, as claimed.
\end{proof}

\subsection{Ore localisation}

In this section we review the notion of Ore localisation, and show how it comes into play when considering various twisted group rings of amenable groups.

\begin{dfn}
 Let $R$ be a ring. A subset $S \subseteq R$ is said to satisfy the \emph{left Ore condition} \iff for all $(p,r) \in R \times S$ there exists $(q,s) \in R \times S$ such that
 \[
  sp = qr
 \]
(this can be interpreted as the existence of left common multiples). The \emph{right Ore condition} is defined analogously by the equation
\[
 ps = rq
\]
The ring $R$ is said to be an \emph{Ore ring} \iff the subset $S$ of $R$ consisting of all non-trivial non-zero-divisors satisfies both the left and the right Ore condition. 
If additionally $R$ has no zero-divisors, then we say that it is an \emph{Ore domain}.
\end{dfn}

\begin{dfn}
Given a ring $R$ and a subset $S \subseteq R \s- \{0\}$ closed under multiplication and satisfying the left and right Ore conditions, we define the \emph{localisation} of $R$ at $S$ by
\[
 RS^{-1} = (R,S) /\sim
\]
where $\sim$ is the transitive closure of the relation identifying $(p,r)$ with $(px,rx)$ for any $x \in S$.

When $S$ is the subset of $R$ consisting of all non-zero non-zero-divisors, we call $R S^{-1}$ the \emph{Ore localisation}.
\end{dfn}

One should think of $(p,r)$ as the right fraction $p/r$. The Ore conditions allow us to change left fractions into right ones and vice versa; they also allow us to find common denominators. Therefore the localisation is itself a ring, and the map $p \mapsto p/1$ is a ring morphism. The construction is explained in Cohn's book \cite[Section 1.2]{Cohn1977}, and in more details in Passman's book~\cite[Section 4.4]{Passman1985}

\begin{rmk}
 \label{Ore rmk}
 Clearly, the Ore condition can be used to find common denominators for any finite number of elements in the localisation $R S^{-1}$.
\end{rmk}

The following facts may also be found in \cite[Section 1.2]{Cohn1977}.

\begin{prop}
\label{Ore loc}
Let $R$ be an Ore ring, and let $S$ denote the set of non-zero non-zero-divisors.
\begin{enumerate}
\item $R$ embeds into its Ore localisation $ R S^{-1}$.
\item Every automorphism of $R$ extends to an automorphism of $R S^{-1}$.
\label{Ore autos}
\item \label{Ore funct} When $R$ is an Ore domain, then any ring monomorphism $R' \into R$ extends to a monomorphism $R' S'^{-1} \into R S^{-1}$, where $S'$ is the group of units of $R'$.
\item When $R$ is an Ore domain, then $R S^{-1}$ is a skew-field.
\end{enumerate}
\end{prop}

Let us now introduce amenability.

\begin{dfn}[Amenable groups]
 A countable group $G$ is \emph{amenable} \iff there exists a sequence $(F_i)$ of non-empty finite subsets of $G$ (a \emph{F\o lner sequence}), such that for every $g \in G$ we have
 \[
  \frac {\vert F_i \vartriangle g.F_i \vert} {\vert F_i \vert} \longrightarrow 0
 \]
as $i \longrightarrow \infty$ (here $\vartriangle$ denotes the symmetric difference).
\end{dfn}
Note that in the definition above one can easily replace the left translates $g . F_i$ by the right translates $F_i . g$.

In the context of group rings, the Ore condition is (almost) equivalent to amenability of the underlying group, as shown by the author in an appendix to an article of Bartholdi~\cite{Bartholdi2016}.

\begin{thm}[{\cite[Theorem A.1]{Bartholdi2016}}]
Let $G$ be a group, and suppose that the group ring $\Z G$ has no zero-divisors. Then $\Z G$ is an Ore domain \iff $G$ is amenable.
\end{thm}

One of the implications was shown by Tamari~\cite{Tamari1957}. We will go back to his proof, since we will require a version of his theorem for twisted group rings.

\begin{thm}[Tamari~{\cite{Tamari1957}}]
\label{tamari}
 Let $G$ be an amenable group, and let $\K$ be a skew-field. If a twisted group ring $\K G$ does not contain zero-divisors, then $\K G$ is an Ore domain.
\end{thm}
\begin{proof}
 Since $\K G$ admits an anti-automorphism induced by $g \mapsto g^{-1}$, it is enough to check the Ore condition on one side.
 (Note that this anti-automorphism sends $\lambda g$ to $g^{-1} \lambda = \phi(g^{-1})(\lambda) g^{-1}$ for $\lambda \in \K$.)

 Let $p,r \in \K G $ with $r \neq 0$. Let $F$ be a F\o lner set such that
 \[
 \frac {\vert F \vartriangle F.g \vert} {\vert F \vert} < \frac 1 { \vert U \vert}
 \]
 for every $g \in U = \supp p \cup \supp r$. Consider $q = \sum_{f \in F} \lambda_f f$ and $s = \sum_{f \in F} \lambda'_f f$; we treat the elements $\lambda_f, \lambda_f' \in \K$ as $2 \vert F \vert$ variables. We now try to solve the linear equation
 \[
 qr(h) = sp(h)
 \]
 for every $h \in G$. Note that this equation is trivial except possibly for
 \[h \in F \cup \bigcup_{g \in U} F.g\]
  that is except for fewer than $2 \vert F \vert$ elements $h$. Hence we are solving a system of fewer than $2 \vert F \vert$ equations with $2 \vert F \vert$ variables. Since we are working over a skew-field, a non-zero solution exists. Let us pick one such -- this way we have defined $q$ and $s$ such that $qr = sp$. If $s=0$ then $r$ is a zero-divisor, which is a contradiction. Therefore the pair $(q,s)$ is as required.
\end{proof}

For the sake of completeness, let us also introduce the elementary amenable groups.

\begin{dfn}[Elementary amenable groups]
 The class of \emph{elementary amenable} groups is the smallest class containing all finite groups and all abelian groups, and closed under subgroups, quotients, direct unions, and extensions.
\end{dfn}
It is classical that all elementary amenable groups are amenable -- the class of amenable groups contains finite groups and $\Z$, and is closed under subgroups, colimits, quotients, and extensions.

In \cref{sec: vanishing} we will need the following lemma.

\begin{lem}
\label{inheriting Ore}
 Let $G$ be any group, and let $N$ be a normal subgroup. 
 Suppose that $S \subseteq \Z N \s- \{0\}$ is closed under multiplication and satisfies the left Ore condition in $\Z N$. Then $S^G$, the multiplicative closure of the $G$-conjugates of $S$, satisfies the left Ore condition in $\Z G$.
\end{lem}
\begin{proof}
We first argue that $S^G$ satisfies the left Ore condition in $\Z N$, and then we argue that any $G$-invariant subset of $\Z N$ closed under multiplication and satisfying the left Ore condition in $\Z N$ actually also satisfies the left Ore condition in $\Z G$.

To prove the first claim, take $p \in \Z N$ and $s = s_1 \dots s_n$ where each $s_i$ lies in some $G$-conjugate of $S$. We will find left common multiples of $p$ and $s$.
We argue by induction on $n$. If $n=1$ then the claim follows from the Ore condition for $S$ (which clearly holds for all conjugates of $S$ as well). Otherwise, using the Ore condition, there exist $q \in \Z N$ and $t_n$ in the same conjugate of $S$ as $s_n$, such that
\[
t_n p = q s_n
\]
Now by the inductive hypothesis there exist $r \in \Z N$ and $ u \in S^G$ such that
\[
 uq = r s_1\dots s_{n-1}
\]
and therefore
\[
 ut_n p = r s
\]
This proves the first claim.

\smallskip
To prove the second claim, take $s \in S$ and $p \in \Z G$, where we assume that $S$ is multiplicatively closed and $G$-invariant.
We have $\Z G = (\Z N) \, G/N$ (as in \cref{key example}), and so we can write
\[
 p = \sum_{i=1}^n p_i
\]
where $p_i = \kappa_i g_i$, with $g_i \in G/N$ and $\kappa_i \in \Z N$.
Using the Ore condition in $\Z N$ we conclude that there exist $\kappa_1' \in \Z N$ and $s_1 \in S$ such that $s_1 {\kappa_1} =  \kappa_1' s^{{g_1}^{-1}}$, where exponentiation denotes right conjugation, and we are using the fact that $S$ is $G$-invariant.

We now argue by induction on $n$: if $n=1$, then
\[
 {s_1} p = {s_1} \kappa_1 g_1 = \kappa_1' s^{{g_1}^{-1}} g_1 = \kappa_1' g_1 s
\]
and the Ore condition is verified. Otherwise, we have
\[
 s_1 p -  \kappa_1' g_1 s = s_1 p - s_1 p_1 =  \sum_{i=2}^n s_1\kappa_i p_i
\]
and $s_1 \kappa_i \in \Z N$ for every $i$.
By the inductive hypothesis there exist $q \in \Z G$ and $t \in S$ such that
\[
 t \cdot (\sum_{i=2}^n s_1 \kappa_i p_i) = q s
\]
and so
\[
 t s_1 p = (t \kappa_1' g_1 + q)s \qedhere
\]
\end{proof}

\section{Matrices and their polytopes}
\label{sec dets}

In this section, we consider a finitely generated free-abelian group $H$.
With the applications in mind, one will not err by thinking of $H$ as the free part of the abelianisation of a finitely generated group. On the other hand, the twisted group ring $\K H$ (where $\K$ is a skew-field), which will be the main focal point of the section, can be interpreted as a ring of twisted Laurent polynomials in finitely many commuting variables, and the more algebraically-minded reader might prefer this point of view.

\subsection{Polytopes}

We start by looking at polytopes. Note that $H_1(H;\R) = H\otimes_\Z \R$ is a finite dimensional vector space over $\R$.

\begin{dfn}[Polytopes]
 A \emph{polytope} in $H_1(H;\R)$ is a compact subset which is the intersection of finitely many affine halfspaces. Note that, in particular, the empty set is a polytope.

Given a polytope $P$ and a character $\phi \in H^1(H;\R)$, we define 
\[
F_\phi(P) = \big\{ p \in P \mid \phi(p) = \min_{q \in P} \phi(q) \big\}
\]
(clearly, this is precisely the set of points in $P$ on which $\phi$ attains its minimum).

The image $F_\phi(P)$ is also a polytope. The collection $\{ F_\phi(P) \mid \phi \in H^1(H;\R) \}$ is the collection of \emph{faces} of $P$. A face is called a \emph{vertex} \iff it has dimension $0$. Note that $P = F_0(P)$ is also a face of $P$.

A polytope $P$ is \emph{integral} \iff its vertices lie in $H \subseteq H_1(H;\R)$; a polytope $P$ is \emph{$\phi$-flat} \iff $P = F_\phi(P)$.

The \emph{Minkowski sum} of two polytopes $P$ and $Q$ defined by
\[
 P + Q = \{p+q \mid p \in P, q \in Q \}
\]
turns the set of all non-empty polytopes in $H_1(H;\R)$ into a cancellative abelian monoid. It is clear that the Minkowski sum restricts to an operation on the set of non-empty integral polytopes in $H_1(H;\R)$, turning this set into an abelian cancellative monoid as well.

We define $\P(H)$ to be the Grothendieck group of fractions of the monoid of non-empty integral polytopes in $H_1(H;\R)$: the group $\P(H)$ is constructed from the free-abelian group with basis equal to the set of non-empty integral polytopes by factoring out relations given by the Minkowski sum.
It is easy to see that every element in $\P(H)$ is represented by a formal difference $P - Q$ of polytopes. We say that an element is a \emph{single polytope} \iff we can take $Q$ to be a singleton. In this case $P$ is uniquely determined up to translation.
Note that $F_\phi$ induces a homomorphism
$F_\phi \colon \P(H) \to \P(H)$ given by $F_\phi(P-Q) = F_\phi(P) - F_\phi(Q)$.

We also define $\P_T(H)$ to be the quotient of $\P(H)$ by the subgroup consisting of formal differences of singletons. It is immediate that this is equivalent to considering pairs of polytopes $P-Q$ up to translation of the first polytope. Note that a single polytope in $\P(H)$ is represented by a unique single polytope in $\P_T(H)$.
\end{dfn}

Let us now introduce duals, which establish a correspondence between polytopes in $H_1(H;\R)$ and subsets of $H^1(H;\R)$.

\begin{dfn}[Duals]
 Let $P \subset H_1(H;\R)$ be a polytope. Given a face $Q$ of $P$, we define its \emph{duals} to be the connected components of
\[
 \{\phi \in H^1(H;\R) \s-\{0\} \mid F_\phi(P) = Q \}
\]
\end{dfn}
\begin{rmk}
The only case in which a face $Q$ admits more then one dual occurs when $Q = P$ and $P$ is of codimension $1$ in $ H_1(H;\R)$. In this case, there are precisely two duals, each of which consists of a single ray.
\end{rmk}

\begin{rmk}
 Observe also that duals are convex, and if $P$ is not empty, then every dual of a vertex of $P$ is open, and the union of duals of vertices of $P$ is dense in $H^1(H;\R)$.
\end{rmk}

\begin{lem}
\label{faces decomp}
Let $P = \sum_{i=1}^n P_i$ be a Minkowski sum of polytopes, and let $Q$ be a face of $P$. There exist unique faces $Q_i$ of $P_i$ such that $Q = \sum_{i=1}^n Q_i$.
\end{lem}
\begin{proof}
 Let $Q$ be a face of $P$.
By definition, this means that there exists a character $\psi$ such that $Q = F_\psi(P)$. Let $Q_i = F_\psi(P_i)$; we clearly have $Q = \sum Q_i$. We now argue that the definition of $Q_i$ is in fact independent of the choice of $\psi$.

Perturb $\psi$ slightly inside of the dual of $Q$ containing it; this way we can potentially decrease the dimensions of some of the faces $F_\psi(P_i)$ without increasing the dimension of any other. Since $Q$ is well-defined, this proves that the faces $Q_i$ do not depend on the choice of $\psi$ locally. Now, the dual of a face is unique or there are two duals, each consisting precisely of one character, the two characters being antipodal. In the former case, the fact that $Q_i$ is well defined follows immediately from connectivity; 
in the latter case, it is clear that if $Q = F_\psi(P) = F_{-\psi}(P)$ then $F_\psi(P_i) = F_{-\psi}(P_i)$ for each $i$.
\end{proof}

We now state a result of Funke which gives us a method of recognising single polytopes among the elements of $\P(H)$.

\begin{prop}[Funke~{\cite[Lemma 4.3]{Funke2018}}]
\label{funke}
Suppose that $H$ is of rank at least $2$.
For every $P \in \P(H)$, the element $P$ is a single polytope \iff $F_\phi(P)$ is a single polytope for each $\phi \in H^1(H;\Z) \s- \{0\}$.
\end{prop}

The following corollary occurs as \cite[Lemma 4.7]{Funke2018}.
\begin{cor}
\label{flat polys}
Let $P \in \P(H)$. Suppose that for every $\phi \in H^1(H;\Z) \s- \{0\}$ there exists a $\phi$-flat polytope $X_\phi$ such that $P+X_\phi$ is a single polytope. Then $P$ is a single polytope.
\end{cor}
\begin{proof}
 We will argue by induction on the rank of $H$. When this rank is $0$,
 every element of $\P(H)$ is a single polytope. When the rank is $1$, then the only $\phi$-flat polytopes are singletons, and so $P$ is a single polytope.

 Now suppose that the result holds for $n-1$, and let $H$ be of rank $n$ (with $n\geqslant 2$). We will argue that for every $\psi \in H^1(H;\Z) \s- \{0\}$, the face $F_\psi(P)$ is a single polytope; in view of \cref{funke}, this suffices.

Note that $F_\psi(X_\phi)$ is a $\phi$-flat polytope. Also, $F_\psi(P+X_\phi) = F_\psi(P) + F_\psi(X_\phi)$ is a single polytope, since $P + X_\phi$ is.
But, up to translation, we have $F_\psi(P) = Q - Q'$ such that $Q,Q'$ and $F_\psi(X_\phi)$ lie in $\P(\ker \psi)$, and the rank of $\ker \psi$ is lower than that of $H$. Hence, by the inductive hypothesis, $F_\psi(P)$ is a single polytope.
\end{proof}

\subsection{Dieudonn\'e determinant}

Let us start by introducing an important convention: whenever we talk about a matrix, we will explicitly state over which ring it lies. In particular, properties like invertibility will always be taken in the ring over which the matrix is defined. If we want to consider a matrix as a matrix over a larger ring (via extension of scalars), we will use tensor notation. Unless specified otherwise, we tensor over $\Z G$. We will use $M_n(R)$ to denote the ring of $n \times n$ matrices over a ring $R$.

We are now going to introduce the Dieudonn\'e determinant which can be computed for square matrices over a skew-field. We will later show how to associate a polytope to such determinants.

\begin{dfn}
Let $A=(a_{ij}) \in M_n(\D)$, where $\D$ is a skew-field. The \emph{canonical representative} of the Dieudonn\'e determinant $\det^c A \in \D$ is defined inductively as follows:
\begin{enumerate}
\item If $n=1$ then $\det^c A = a_{11}$.
\item If the last row of $A$ consists solely of zeros, then $\det^c A = 0$.
\item If $a_{nn} \neq 0$ then we form an $(n-1) \times (n-1)$ matrix $A' = (a'_{ij})$ by setting
$a'_{ij} = a_{ij} - a_{in} a_{nn}^{-1} a_{nj}$,
and declare $\det^c A = \det^c A' \cdot a_{nn}$.
\item If $a_{nn} =0$ and not every entry in the least row of $A$ is zero then let $j$ be maximal such that $a_{nj} \neq 0$. Let $B$ be the symmetric matrix interchanging $j$ and $n$. We declare $\det^c A = -\det^c (AB)$.
\end{enumerate}

The \emph{Diedonn\'e determinant} $\det A$ is defined to be the image of $\det^c A$ in \[\D^\times /[\D^\times,\D^\times] \sqcup \{0\}\]
\end{dfn}

The procedure in step $(3)$ corresponds to multiplying $A$ on the right by elementary matrices until the last row has a single non-zero element. Thus, computing the canonical representative of the Dieudonn\'e determinant consists of putting the matrix $A$ into an upper-diagonal form, and then taking the product of the diagonal entries. Hence it is clear that over commutative fields the Dieudonn\'e determinant coincides with the usual determinant.

Note that when writing $\det A$ we do not need to specify the skew-field, since we have adopted the convention that every matrix comes with a specified ring over which it lies.

\begin{thm}[Dieudonn\'e~{\cite{Dieudonne1943}}]
For any $n$, the Dieudonn\'e determianant \[\det \colon M_n(\D) \to \D^\times /[\D^\times,\D^\times] \sqcup \{0\}\] is multiplicative.
\end{thm}

\subsection{The Newton polytope of a matrix}

Recall that $\K$ is a skew-field, and $\K H$ is a twisted group ring of $H$, a finitely generated free-abelian group.
We start by introducing Newton polytopes of elements in $\K H$.
Note that $\K H$ has no zero-divisors -- this follows directly from \cref{malcev-neumann}, since $H$ is biorderable.

To study the Newton polytopes let us first introduce minima $\mu_\phi$, which are algebraic counterparts to the face maps $F_\phi$.

\begin{dfn}[$\mu_\phi$]
\label{minima}
For any character $\phi \in H^1(H;\R)$ we define
$\mu_\phi \colon { \K H} \to \K H$ by setting
\[
 \mu_\phi(x)(h) = \left\{ \begin{array}{cl}
                           x(h) & \textrm{ if } \phi(h) \textrm{ is minimal in } \phi(\supp x) \\
                           0 & \textrm{ otherwise}
                          \end{array} \right.
\]
(we are using the function notation here).
\end{dfn}
Note that $\supp \mu_\phi(x)$ consists of elements in $H$ with the same value under $\phi$.

Since $\K H$ has no zero-divisors, it is an easy exercise to see that $\mu_\phi$ is multiplicative.

\begin{dfn}
 Let $p \in \K H$ be an element. The associated \emph{Newton polytope} $P(p)$ is the convex hull of $\supp p$ taken in the $\R$-vector space $H_1(H;\R)$.
\end{dfn}

Note that, for every $\phi \in H^1(H;\R)$ and every $p \in \K H$, we have
\[
 F_\phi\big(P(p)\big) = P\big( \mu_\phi(p) \big)
\]

\begin{lem}
\label{P is hom}
The map $P \colon \K H \s- \{0\} \to \P(H)$ satisfies
\[
 P(pq) = P(p) + P(q)
\]
for every $p,q \in \K H \s- \{0\}$
\end{lem}
\begin{proof}
The proof is an induction on the rank $n$ of $H$. If $n=0$, then both sides of the desired equation are trivial. If $n=1$, then $\K H$ is a twisted Laurent polynomial ring in one variable, and the result is immediate.

Now suppose that $n\geqslant 2$. Take any character $\phi \in H^1(H;\R) \s- \{0\}$. We have
\[
 F_\phi\big( P(pq) \big) = P\big( \mu_\phi(pq) \big) = P\big( \mu_\phi(p) \cdot \mu_\phi(q) \big)
\]
Now we use the inductive hypothesis (we might have to translate the polytopes $P\big(\mu_\phi(p)\big)$ and $P\big( \mu_\phi(q) \big)$ first, so that they both lie in the same hyperspace of $H$). We obtain
\[
P\big( \mu_\phi(p) \cdot \mu_\phi(q) \big) =  P\big( \mu_\phi(p)\big) + P\big( \mu_\phi(q) \big) = F_\phi\big( P(p) \big) + F_\phi\big( P(q) \big) = F_\phi\big( P(p) + P(q) \big)
\]
Hence, for every non-trivial $\phi$, we have \[F_\phi\big( P(pq) - P(p) - P(q)\big)=0= {F_\phi\big( -P(pq) + P(p) + P(q)\big)}\]
But then \cref{funke} tells us that $P(pq) - P(p) - P(q)$ and $-\big(P(pq) - P(p) - P(q)\big)$ are single polytopes, which is only true when $P(pq) - P(p) - P(q)$ is a singleton. Using any of the maps $F_\phi$ we immediately see that this singleton is precisely the origin of $H_1(H;\R)$, and therefore $P(pq) = P(p)+P(q)$.
\end{proof}

Since $H$ is amenable and $\K H$ has no zero-divisors, $\K H$ is an Ore domain by \cref{tamari}; let $\D$
 denote the Ore localisation of $\K H$ -- recall that $\D$ is a skew-field containing $\K H$, and elements of $\D^\times$ are fractions of the form $pq^{-1}$ with $p,q \in \K H \s- \{0\}$.

Thanks to \cref{P is hom}, we immediately see that the map $P$ induces a homomorphism $P \colon \D^\times \to \P(H)$ defined by
\[
 P(pq^{-1}) = P(p) - P(q)
\]
Since $\P(H)$ is abelian, the homomorphism $P$ gives a well-defined group homomorphism
\[\D^\times / [\D^\times, \D^\times] \to \P(H)\]
 Hence, for every square matrix $A$ over $\Z H$, we may talk about $P(\det A \otimes \D) \in \P(H) \sqcup \{ \emptyset \}$ (setting $P(0) = \emptyset$).
\begin{dfn}[Newton polytope]
Given a square matrix $A$ over $\K H$, we define its \emph{Newton polytope} to be $P(A) = P(\det A\otimes \D) \in \P(H) \sqcup \{ \emptyset\}$.
\end{dfn}

We are now ready for the first result of the article.

\begin{thm}[Single polytope]
\label{single poly}
 Let $A$ be a square matrix over $\K H$. Then $P(A)$ is empty or a single polytope.
\end{thm}
\begin{proof}
Take $\phi \in H^1(H;\Z) \s- \{0\}$, and let $L = \ker \phi$. Take $z\in H$ with $\phi(z) = 1$, the generator of $\im \phi = \Z$.
Let $\K L$ denote the subring of $\K H$ consisting of elements of support lying in $L$. It is clear that $\K L$ is a twisted group ring of $L$ with coefficients $\K$, and hence an Ore domain; let us denote the Ore localisation of $\K L$ by $\LL$. The localisation $\LL$ embeds into the Ore localisation $\D$ of $\K H$ by \cref{Ore loc}(\ref{Ore funct}), and the action by conjugation  of $z$ on $\K L$ extends to an action on $\LL$ by \cref{Ore loc}(\ref{Ore autos}). Thus, we have an embedding of the twisted group ring $\LL \Z$ (with $\Z$ generated by $z$) into $\D$.


We will think of $\LL \Z$ as a twisted Laurent polynomial ring with variable $z$. We apply Euclid's algorithm over $z$ to $A \otimes \LL \Z$: using only elementary matrices over $\LL \Z$ we put $A \otimes \LL \Z$ into an upper-triangular form. Thus $\det A \otimes \D$ can be represented by $\sum_{i=-n}^n \kappa_i z^i$, a Laurent polynomial in $z$ with coefficients in $\LL$. Since $\LL$ is the Ore localisation of $\K L$,
 there exist $\mu_i, \nu \in \K L$ such that $\kappa_i =  \nu^{-1} \mu_i$ for every $i$ (see \cref{Ore rmk}). Hence $\det A \otimes \D$ is represented by
$\nu^{-1} \sum_{i=-n}^n \mu_i z^i$. If $P(A) \neq \emptyset$, it immediately follows that
\[
 P(A) = P(\det A \otimes \D ) = P(\sum_{i=-n}^n \mu_i z^i) - P (\nu)
\]
Crucially, $P(\sum_{i=-n}^n \mu_i z^i)$ and $P (\nu)$ are single polytopes, and $P(\nu)$ is $\phi$-flat. Thus \cref{flat polys} tells us that $P(A)$ is a single polytope.
\end{proof}

\subsection{Novikov rings}

To see why one should be interested in Newton polytopes of matrices over group rings, we need to introduce the Novikov rings.


\begin{dfn}[Truncated support]
Let $R$ be a ring and $G$ a group. Let $x \in R^G$ be a function.
Given a character $\phi \in H^1(G;\R)$ and a constant $\kappa \in \R$ we define the \emph{truncated support} to be
\[
 \supp_{\phi,\kappa} x = \{ g \in G \mid x(g) \neq 0 \textrm{ and } \phi(g) \leqslant \kappa \}
\]
\end{dfn}

\begin{dfn}[Novikov ring]
Given a character $\phi \in H^1(G;\R)$ we define
\[
 \widehat {R G}^\phi = \{ x \colon G \to R \mid \supp_{\phi,\kappa} x \textrm{ is finite for every } \kappa \in \R \}
\]
Pointwise addition turns $\widehat {R G}^\phi$ into an abelian group; to endow it with a ring structure we use the twisted convolution \eqref{twisted conv} -- this way the set-theoretic inclusion $R G \subseteq \widehat {R G}^\phi$ turns into an embedding of rings.

For any $S \subseteq H^1(H;\R)$ we define $\widehat {R G}^S = \bigcap_{\phi \in S} \widehat {R G}^\phi$.
\end{dfn}

We will treat $\widehat {R G}^\phi$ as a left $R G$ module, and a right $\widehat {R G}^\phi$ module, so we will tensor $R G$ modules with $\widehat {R G}^\phi$ on the right.

The Novikov rings play a crucial role in the Bieri--Neumann--Strebel invariants via the theorem of Sikorav (\cref{sikorav}), which we will discuss later. The key technical feature of the current article is that we will discuss $\Sigma$-invariants of matrices.

\subsection{\textSigma-invariants of matrices}

\begin{dfn}[$\Sigma(A)$]
Let $R$ be a ring, $G$ a group, and $R G$ a twisted group ring.
 Let $A$ be a (not necessarily square) matrix over $R G$. We define $\Sigma(A) \subseteq H^1(G;\R)$, the \emph{$\Sigma$-invariant} of $A$, by declaring $\phi \in  \Sigma(A)$ \iff $A \otimes \widehat {R G }^\phi$ admits a right inverse.
\end{dfn}

Recall our convention: $A \otimes \widehat {R G }^\phi$ admitting a right inverse means precisely that the inverse lies over  $\widehat {R G }^\phi$ as well.

\begin{dfn}[$\phi$-identity]
Let $\phi \in H^1(G;\R)$ be given. An $n \times m$ matrix $A$ over $\widehat {R G}^\phi$ is said to be a \emph{$\phi$-identity} \iff for every entry $x$ of $A - \I$ we have $\phi(\supp x) \subseteq (0,\infty)$. (Here, $\I$ denotes the identity matrix extended by a zero matrix either on the right or at the bottom, so that $\I$ is an $n \times m$ matrix.)
\end{dfn}

For the purpose of the above definition, we treat elements in $\widehat {R G}^\phi$ as $1 \times 1$ matrices.

\begin{lem}
\label{phi id}
Suppose that a square matrix $A$ over $R G$ is a $\phi$-identity for some ${\phi \in H^1(G;\R)}$. There exists an open neighbourhood $U \subseteq H^1(G;\R)$ of $\phi$ such that $A \otimes \widehat {R G}^U$ is right and left invertible.
\end{lem}
\begin{proof}
Let $B$ be defined by  the equation $A = \I - B$. Observe that the definition of $\phi$-identity tells us that $\phi$ is strictly positive on the supports of the entries of $-B$, and hence of $B$.

Since the supports of the entries of $B$ are finite, there exist $\kappa>0$ and an open neighbourhood $U$ of $\phi$ in $H^1(G;\R)$ such that for every $\psi \in U$ the supports of the entries of $B$ are sent by $\psi$ to $(\kappa,\infty)$. Thus
 \[
  Y = \sum_{i=0}^\infty B^i
 \]
defines a matrix over $\widehat {R G}^U$, as the supports of the entries of $B^i$ are mapped by every $\psi \in U$ to $(i \kappa, \infty)$.
Clearly $AY = (\I - B)Y = \I$ and $YA = Y(\I-B) = \I$.
\end{proof}

\begin{lem}
\label{entries far away}
\label{bns open}
 Let $A$ be an $n \times m$ matrix over $R G$, and let $\phi \in \Sigma(A)$. There exists an $m \times n$ matrix $X$ over $R G$ such that $AX$ is a $\phi$-identity. Moreover,
 there exists an open neighbourhood $U$ of $\phi$ in $H^1(G;\R)$ such that $A\otimes \widehat{R G}^U$ admits a right inverse.
\end{lem}
\begin{proof}
We let $\supp A$ denote the union of the supports of the entries of $A$.
Take $C \in \N$ such that
\[
 C > - \min \phi(\supp A)
\]
Since $A \otimes \widehat {R G}^\phi$ is right-invertible, there exists a matrix $B = (b_{ij})$ over $\widehat {R G}^\phi$ such that $AB = \I$. Let a pair $(i,j)$ be fixed for a moment.
We have $b_{ij} = \sum \lambda_g g$ (where the sum is typically infinite). We set
\[
 b_{ij}^0 = \sum_{\phi(g) < C} \lambda_g g  \in R G                                                                                                                                                                                                                                                                                                                 \]
and define $b_{ij}^+ \in \widehat {R G}^\phi$ by the equation $b_{ij} = b_{ij}^0 + b_{ij}^+$. We now apply the procedure to every pair $(i,j)$, and define matrices $B^0 = (b_{ij}^0)$ and $B^+ = (b_{ij}^+)$; observe that we have $B =B^0 + B^+$.

By the choice of $C$,  the supports of the entries of $A B^+$ are mapped by $\phi$ to $(0,\infty)$. But $\I = AB = AB^0 + AB^+$, and so $AB^0$ is a $\phi$-identity. We set $X = B^0$.

The second assertion follows from \cref{phi id}, observing that $AX$ is a square matrix.
\end{proof}

Note that, in particular, the above result implies that $\Sigma(A)$ is an open subset of $H^1(G;\R)$, as one would expect.

\smallskip

As before, let $\K$ be a skew-field, $\K H$ a twisted group ring, and
 $\D$ the Ore localisation of $\K H$. In order to study the Novikov rings $\widehat {\K H}^\phi$, we will use the minima $\mu_\phi$: recall that we introduced them (in \cref{minima}) as maps $\K H \to \K H$. It is however immediate that the definition extends and defines multiplicative maps $\mu_\phi \colon \widehat {\K H}^\phi \to \K H$.

 In what follows, all tensoring takes place over $\K H$.


\begin{lem}
\label{malcev-neumann H}
For every $\phi \in H^1(H;\R)$ there exists a skew-field $\mathcal F$ such that $\widehat {\K H}^\phi$ and $\D$ both embed into $\mathcal F$ in such a way that the embeddings agree on $\K H$.
\end{lem}
\begin{proof}[Sketch proof]
 The skew-field $\mathcal F$ will be the Malcev--Neumann skew-field, which we discussed in \cref{malcev-neumann}; a similar embedding was constructed in \cite[Lemma 5.5]{FunkeKielak2018}.

 Let $\leqslant$ be a biordering of $H$ which makes $\phi \colon H \to \R$ into an order-preserving homomorphism. Then $\mathcal F$, the subset of $\K^H$ consisting of functions with supports well-ordered with respect to $\leqslant$, forms a skew-field. Now $\widehat {\K H}^\phi$ is actually a subset of $\mathcal F$.
 This also implies that $\K H$ is a subset of $\mathcal F$; let us denote the resulting embedding by $\iota$.

 It is easy to see that the embedding $\iota \colon \K H \into \mathcal F$ induces an embedding of the Ore localisation of $\K H$, that is of $\D$ -- we simply set $\iota(pq^{-1}) = \iota(p) \iota(q)^{-1}$. The injectivity of this extended $\iota$ is immediate.
\end{proof}

We obtain the following, immediate corollary.

\begin{cor}
\label{no zero divs}
The ring $\widehat {\K H}^\phi$ has no zero-divisors.
\end{cor}


Recall that we have implicit structural functions in the definition of $\K H$. The same functions (via formula \eqref{twisted conv}) are used to define multiplication in $\widehat {\K H}^\phi$; in fact they can be also used to define the multiplications \[\K H \times \K^H \to \K^H \textrm{ and } \K^H \times \K H \to \K^H\]

\begin{dfn}
We say that an element $x\in \D$ is \emph{represented} by $y \in \K^H$ \iff $x=pq^{-1}$ with $p,q \in \K H$ and $p = yq$.
\end{dfn}

Note that $x$ can be represented by different elements in $\K^H$, since $\K^H$ contains zero-divisors with respect to the twisted convolution. The situation is however different if $x$ is represented by an element of $\widehat{\K H}^\phi$.

\begin{lem}
\label{bns unique}
Let $x \in \D$ be an element represented by $y$ and $y'$, both lying in $\widehat{\K H}^\phi$ for some $\phi \in H^1(H;\R)$. Then $y=y'$.
\end{lem}
\begin{proof}
By definition, we have $x = pr^{-1} = p'r'^{-1}$ with $p,r,p',r' \in \K H$ and $r,r' \neq 0$. We also have
$p = yr$ and $p'=y'r'$. Now, using the Ore condition, there exist $q,s \in \K H$ such that $s\neq 0$ and $sp= qr$.
Now $pr^{-1} = p'r'^{-1}$ means precisely that $sp'=qr'$ as well.
Thus
\[
 qr' = sp' = sy'r'
\]
and so $(q-sy')r' = 0$. As $r' \neq 0$ and $\widehat{\K H}^\phi$ contains no zero-divisors, we conclude that $q = sy'$. Therefore
\[
syr = sp = qr = sy'r
\]
and arguing as above (using $s,r \neq 0$) we see that $y - y' = 0$.
\end{proof}

Because of the above lemma, we will say that $x\in \D$ is equal to $y \in \widehat{\K H}^\phi$ \iff it is represented by $y$.


\begin{lem}
\label{unique RI}
Let $A$ be a square matrix over $\K H$, such that $A \otimes \D$ is invertible.
Let $C\subseteq H^1(H;\R)$ be a non-empty subset. If $A\otimes \widehat{\K H}^{C}$ admits a right-inverse, then this inverse is unique.
\end{lem}
\begin{proof}
Take $\phi \in C$.
Let $\mathcal F$ denote the skew-field from \cref{malcev-neumann H}, which contains $\widehat {\K H}^\phi$ and $\D$ as subrings. We will carry all matrix multiplication out in $\mathcal F$.

Suppose that we have two right-inverses, $X$ and $X'$.  Then
\[A \cdot (X - X') = 0\]
But we have assumed that $A\otimes \D$ is invertible, and so there exists a matrix $Y$ over $\mathcal{F}$ such that $Y \cdot A = \I$. Hence
\[
X - X' =Y\cdot A  \cdot (X  - X' ) =  Y \cdot 0 = 0
\]
and so $X - X' = 0$.
\end{proof}

\begin{cor}
\label{bns union}
Let $A$ be a square matrix over $\K H$, such that $A \otimes \D$ is invertible. Let $X$ be a right inverse of  $A\otimes \widehat{\K H}^C$ and $X'$ be a right inverse of $A \otimes \widehat{\K H}^{C'}$ for some $C,C' \subseteq H^1(H;\R)$ with $C \cap C' \neq \emptyset$. Then $X = X'$, and $X$ is a right inverse of $A\otimes \widehat{\K H}^{C \cup C'}$.
\end{cor}
\begin{proof}
Let $\phi \in C \cap C'$.
By assumption, there exist right inverses $X$ and $X'$ of, respectively, $A\otimes \widehat{\K H}^C$ and $A \otimes \widehat{\K H}^{C'}$.
In particular, both $X$ and $X'$ can be viewed as matrices over $\widehat{\K H}^\phi$, and so $X = X'$ by \cref{unique RI}. Now $X$ is a matrix over $\widehat{\K H}^{C} \cap \widehat{\K H}^{C'} =\widehat{\K H}^{C \cup C'}$.
\end{proof}

\begin{lem}
\label{bns convex}
 Let $S \subseteq H^1(H;\R)$ be any subset, and let $C$ denote its convex hull. Then
 \[
  \widehat{\K H}^S = \widehat{\K H}^C
 \]
\end{lem}
\begin{proof}
Let $x \in  \widehat{\K H}^S$ be any element. Take $\phi_1, \dots, \phi_k \in S$, and $t_1, \dots, t_k \in [0,1]$ with $\sum_{i=1}^k t_i = 1$. Let $\psi = \sum_{i=1}^k t_i \phi_i$. It is enough to show that $x \in \widehat{\K H}^\psi$.

Pick $\kappa \in \R$ and for each $\rho \in H^1(H;\R)$ consider the truncation $x_{\rho,\kappa}$ defined by
\[
 x_{\rho,\kappa}(h) = \left\{ \begin{array}{cl}
                        x(h) & \textrm{if } \rho(h) \leqslant \kappa \\
                        0 & \textrm{otherwise}
                       \end{array} \right.
\]
We need to show that $\supp x_{\psi,\kappa}$ is finite for every $\kappa$. To this end, take $h \in \supp x_ {\psi,\kappa}$. If $\phi_i(h) > \kappa$ for all $i$, then the same is true for $\psi(h)$, which is a contradiction. So there exists $i$ such that $\phi(h) \leqslant \kappa$. Thus $h$ lies in the support of $x_{\phi_i, \kappa}$. So $\supp x_{\psi, \kappa}$ is contained in the union of the supports of the elements $x_{\phi_i, \kappa}$. But this union is finite.
\end{proof}

Let us summarise  the results obtained so far.

\begin{prop}
\label{bns summary}
Let $A$ be a square matrix over $\K H$, such that $A \otimes \D$ is invertible. Every connected component $C$ of $\Sigma(A)$ is open and convex, and $A \otimes \widehat{\K H}^C$ is right-invertible.
\end{prop}
\begin{proof}
For every $\phi \in C$, by \cref{bns open}, there exists an open neighbourhood $U_\phi$ of $\phi$ in $\Sigma(A)$ such that $A \otimes \widehat{\K H}^{U_\phi}$ admits a right-inverse $X_\phi$. We may assume $U_\phi$ to be connected.
Since $C$ is a connected component, it is immediate that $U_\phi \subseteq C$ for every $\phi$.

 Consider the transitive closure $\sim$ of the relation on $\mathcal U = \{ U_\phi \mid \phi \in C\}$ given by having non-empty intersection. The union of all sets $U_\phi$ belonging to an equivalence class of $\sim$ is open, and two such unions for two distinct equivalence classes are disjoint. Since $C$ is connected, $\sim$ admits only one equivalence class. Therefore for any two sets $U_\phi$ and $U_\psi$ in $\mathcal U$ there exists a finite sequence of sets $U_\phi = U_0, U_1, \dots, U_n = U_\psi$ in $\mathcal U$ such that $U_i \cap U_{i+1} \neq \emptyset$ for every $i$. Let $X_i$ denote the matrix corresponding to $U_i$. \cref{bns union} implies that $X_i = X_{i+1}$ for every $i$, and therefore $X_\phi = X_\psi$. Thus, letting $X = X_\phi$ for any $\phi \in C$, we see that $X$ lies over $\bigcap_\phi \widehat{\K H}^{U_\phi} = \widehat{\K H}^{\bigcup_\phi U_\phi} = \widehat{\K H}^C$.

The connected component $C$ is convex by \cref{bns convex} and open by the discussion above (or by \cref{bns open}).
\end{proof}

Before stating the main theorem of this section let us look at the following lemma, whose central purpose is to elucidate the statement of the theorem, as well as some arguments used in its proof.

\begin{lem}
\label{inv of novikov elements}
 Let $x \in \K H$ be any element, let $\phi \in H^1(H;\R)$ be any character, and let $C$ the unique dual of a face of $P(x)$ containing $\phi$. The following are equivalent.
 \begin{enumerate}
 \item The element $x$ is invertible in $\widehat {\K H}^C$.
  \item The element $x$ is invertible in $\widehat {\K H}^\phi$.
  \item The support $\supp \mu_\phi(x)$ is a singleton.
  \item The dual $C$ is a dual of a vertex of $P(x)$.
 \end{enumerate}
\end{lem}
\begin{proof}
\noindent \textbf{(1)$\Rightarrow$(2)}
This is immediate, since $\widehat {\K H}^C \leqslant \widehat {\K H}^\phi$ by definition.

\medskip
\noindent \textbf{(2)$\Rightarrow$(3)}
There exists $z \in \widehat {\K H}^\phi$ such that $xz=1$. Using the minima we see that
\[
 \mu_\phi(x) \mu_\phi(z) = 1
\]
as well. But the minima lie in $\K H$, and the units in $\K H$ have support of size $1$ by \cref{units of biord}. Thus $\mu_\phi(x)$ is supported on a singleton.

\medskip
\noindent \textbf{(3)$\Rightarrow$(4)}
We have $F_\phi\big( P(x) \big) = P\big( \mu_\phi(x) \big)$, which is a singleton. Thus,
$\phi$ lies in a dual of a vertex of $P(x)$, by the very definition of dual. But this dual is $C$, and so $C$ is as claimed.

\medskip
\noindent \textbf{(4)$\Rightarrow$(1)}
For every $\psi \in C$, we have $F_\psi\big( P(x) \big)$ equal to a singleton, say $h \in H$.
Thus $\mu_\psi(x) = \lambda h$ with $\lambda \in \K \s- \{0\}$. Therefore $\mu_\psi(x)$ is invertible in $\K H$, and $y = \mu_\psi(x)^{-1} x$ is an element of $\K H$ which is a $\psi$-identity. Hence $y$, and therefore also $x$, is invertible over $\widehat {\K H}^{U_\psi}$, where $U_\psi$ is some open neighbourhood  of $\psi$, by \cref{phi id}. Since the open sets $U_\psi$ cover $C$, we have $C\subseteq \Sigma(x)$. Therefore $x$ is invertible in $\widehat {\K H}^C$ by \cref{bns summary} (applied to the $1\times 1$ matrix $x$).
\end{proof}

We may view $x$ above as a $1 \times 1$ matrix; in this case we have $P(x) = P(\det x \otimes \D)$, and the invertibility of $x$ over $\widehat {\K H}^\phi$ is equivalent to $\phi \in \Sigma(x)$. The following result generalises this to square matrices of arbitrary size, and constitutes the second main result of the article.

\begin{thm}[$\Sigma$-invariants of matrices]
\label{K-bns matrix}
Let $\D$ denote the Ore localisation of $\K H$, and suppose that $A$ is a square matrix over $\K H$ with $A \otimes \D$ invertible. Let $\phi \in H^1(H;\R)$ be a character, and let $C$ denote the unique dual of a face of $P(A)$ containing $\phi$. The following are equivalent.
 \begin{enumerate}
 \item The matrix $A\otimes \widehat {\K H}^C$ is right-invertible.
  \item We have $\phi \in \Sigma(A)$.
  \item The dual $C$ is a dual of a vertex of $P(A)$.
 \end{enumerate}
\end{thm}
\begin{proof}
\noindent \textbf{(1)$\Rightarrow$(2)}
This is immediate, since $\widehat {\K H}^C \leqslant \widehat {\K H}^\phi$ by definition, and $\phi \in \Sigma(A)$ means precisely that $A\otimes \widehat {\K H}^\phi$ is right-invertible.

\medskip
\noindent \textbf{(2)$\Rightarrow$(3)}
Set $P = P(A) = P(\det A \otimes \D)$. Recall that this is a single polytope by \cref{single poly} -- it cannot be empty, since $A\otimes \D$ is invertible and so its Dieudonn\'e determinant is not zero.
 We start by taking $\phi \in \Sigma(A)$, and claiming that $C$ is a dual of a vertex of $P$.

 By \cref{entries far away}, there exists a matrix $X$ over $\K H$ such that $AX$ is a $\phi$-identity. Set $AX = (x_{ij})$.
Consider the last row of $AX$. Its last entry $x_{nn}$ is invertible over $\widehat {\K H}^\phi$ by \cref{phi id}, since it is itself a $\phi$-identity. Thus, we may multiple $AX$ on the right by a number of elementary matrices in such a way that the resulting matrix $A'$ has zeroes in the last row, with the exception of the last entry (which remains unchanged).
Observe that we may use elementary matrices whose non-zero off-diagonal entries are precisely the elements $-x_{nn}^{-1}x_{ni}$ with $i \in \{1, \dots, n-1\}$.

Each of these off-diagonal entries is a product of an element of $\K H$ and an inverse of such an element, the inverse taken in $\widehat{ \K H}^\phi$. Thus, we may see each of these entries as lying over $\D$, since $\D$ contains all inverses of elements in $\K H$.
Moreover, the support of such an off-diagonal entry is taken by $\phi$ into $(0,\infty)$, since $AX$ is a $\phi$-identity, and therefore $\phi(\supp x_{nn}) \subset [0,\infty)$ and $\phi(\supp x_{in}) \subset (0,\infty)$, implying that $\phi(\supp x_{nn}^{-1} x_{in}) \subset (0,\infty)$ as well.
It is easy to see that multiplying a $\phi$-identity by elementary matrices whose non-zero off-diagonal entries are as above yields another $\phi$-identity.
Therefore $A'$ is a $\phi$-identity.

Repeating the process for other rows, we conclude the existence of a matrix $Y$ over $\D$, with $\det Y = 1$, and such that $AXY$ is upper triangular, and the diagonal entries are $\phi$-identities.
Hence $\det AX\otimes \D = \det AXY\otimes \D$ can be represented by an element $x \in \widehat{ \K H}^\phi$ which is also a $\phi$-identity.

We can also view $x$ as an element in $\D$, that is as a fraction of elements $p,q \in \K H$ (since $\D$ is the Ore localisation of $\K H$). Thus we have $x = p q^{-1}$, and so
\[
 xq = p
\]
Since $p$ and $q$ have finite support in $H$, we can carry out the multiplication above in $\widehat {\K H}^\phi$. Since $x$ is a $\phi$-identity, we have $\mu_\phi(x) = 1$, and therefore $\mu_\phi(q) = \mu_\phi(p)$.
Thus we have
\[
 F_\phi(P(p)) = F_\phi(P(q))
\]
But we know that $P(\det AX \otimes \D)$ is a single polytope by \cref{single poly}, and so $F_\phi\big(P(\det AX \otimes \D)\big) = F_\phi(P(p)) - F_\phi(P(q))$ is a singleton. We have \[P(AX) = P(A) + P(X)\]since $P\colon \D^\times \to \P(H)$ is a homomorphism. Both $P=P(A)$ and $P(X)$ are single polytopes, again using \cref{single poly}.
Therefore so are $F_\phi\big( P(A) \big)$ and $F_\phi\big( P(X) \big)$, and their sum is a singleton. This is only possible when both $F_\phi\big( P(A) \big)$ and $F_\phi\big( P(X) \big)$ are singletons. In particular, $F_\phi\big( P(A) \big)$ is a singleton, which means precisely that $C$ is a dual of a vertex of $P(A)$ (as $\phi \in C$). This proves the claim.

\medskip
\noindent \textbf{(3)$\Rightarrow$(1)}
Now let $C$ be a dual of a vertex of $P$. We aim at showing that $A\otimes \widehat{\K H}^C$ admits a right-inverse.

Since $C$ is open, there exists $\phi \in C \cap H^1(H;\Z)$. Let us take such a $\phi$. We will proceed in $3$ steps.

\step{1} We claim that there exists a $\phi$-flat polytope $Q$, such that for every intersection $C_0$ of a dual of a vertex of $Q$ with $C$, the matrix $A\otimes \widehat{ \K H}^{C_0}$ is right-invertible.


To prove the claim we start by introducing some extra notation (which is exactly the same as the one used in the proof of \cref{single poly}).
Let $L = \ker \phi$. The ring $\K H$ embeds into $\LL \Z$, where $\LL$ is the Ore localisation of $\K L$, and the embedding is induced by $\phi \colon H \to \Z$.
We now use Euclid's algorithm (treating $\LL \Z$ as a Laurent polynomial ring in one variable): it gives us a matrix $X$ over $\LL \Z$ (a product of elementary matrices) such that $A\otimes \LL \Z \cdot X = Y$ and $Y$ is an upper-triangular matrix over $\LL \Z$. Also, $\det X \otimes \D = 1$.

Let us focus first on $X$. It is a matrix over $\LL \Z$, and $\LL$ is the Ore localisation of $\K L$. Thus, there exists an element $p \in \K L$ such that for every entry $x$ of $X$ we have $px \in \K L$ (see \cref{Ore rmk}). This implies that for every $\zeta \in H^1(H;\R)$, if $p$ is invertible in $\widehat {\K H}^\zeta$ then every entry of $X$ can be represented by an element in $\widehat {\K H}^\zeta$ (in this case, we will simply say that $X$ lies over $\widehat {\K H}^\zeta$).
In fact more is true: if $C'$ denotes a dual of some vertex of $P(p)$, then $X$ lies over $\widehat {\K H}^{C'}$, as \cref{inv of novikov elements} tells us that $p$ is invertible in $\widehat {\K H}^{C'}$.

Now let us look at $Y=(y_{ij})$. Since $Y$ is upper-triangular, we have $\det Y\otimes \D$ represented by $\prod_i y_{ii}$. 
Therefore
\[
 P=P(\det A \otimes \D) = P(\det Y \otimes \D) = \sum_i  P(y_{ii})
\]
Take $\psi \in C$.
Since $\phi, \psi \in C$, we have $F_\phi(P) = F_\psi(P)$, and so
\[
 \sum_i F_\phi\big( P(y_{ii}) \big) = F_\phi(P) = F_\psi(P) = \sum_i F_\psi\big( P(y_{ii}) \big)
\]
Now \cref{faces decomp} implies that 
\[
 F_\psi P(y_{ii}) =  F_\phi P(y_{ii})
\]
for every $i$, since the decomposition of a face of $P$ into a sum of faces of the polytopes $P(y_{ii})$ is unique.

Recall that $y_{ii}$ lies in $\LL \Z$ for every $i$.
Let us now argue as for the matrix $X$ before: we see that for every $i$ there exists $q_i \in \K L$ such that \[q_i y_{ii} \in (\K L) \Z\] For notational convenience, let us fix $i$ for the moment. We then have
\[q_i y_{ii} = (r_0 + r_1 z + \dots + r_m z^m)z^n\]
for some $m \in \N$ and $n \in \Z$, where $z$ stands for a generator of $\im \phi = \Z$, the coefficients $r_i$ lie in $\K L$, and $r_0 \neq 0 \neq r_m$.
We have
\[
 \mu_\phi (q_i y_{ii}) = r_0 z^n
\]
Since $q_i \in \K L$, we also have $\mu_\phi(q_i) = q_i$, and so
\[
 q_i \cdot \mu_\phi ( y_{ii}) = r_0 z^n
\]
Applying the polytope map and using that $F_\psi P(y_{ii}) = F_\phi P(y_{ii})$ we obtain
\[
 P(q_i) + F_\psi\big( P(y_{ii}) \big) = P(r_0 z^n)
\]
It is now clear that replacing $P(q_i)$ by $F_\psi\big( P(q_i) \big)$ on the left-hand side results in obtaining a polytope on the right-hand side which is a subset of $P(r_0 z^n)$. But we also have
\[
 F_\psi\big(P(q_i)\big) + F_\psi\big( P(y_{ii}) \big) = F_\psi\big( P((r_0 + r_1 z + \dots + r_m z^m)z^n) \big)
\]
We conclude that $F_\psi\big( P((r_0 + r_1 z + \dots + r_m z^m)z^n) \big)$ is a subset of $P(r_0 z^n)$, which implies that $\mu_\psi((r_0 + r_1 z + \dots + r_m z^m)z^n) = \mu_\psi(r_0 z^n)$.

Let $C''$ denote the non-empty intersection of $C$ with a dual of some specified vertex of $P(r_0)$, and take $\rho \in C''$. Since $\rho$ lies in a dual of a vertex of $P(r_0)$, and hence also its translate $P(r_0z^n) = P(r_0) + P(z^n)$, \cref{inv of novikov elements} tells us that $\mu_\rho(r_0 z^n)$ is supported on a singleton. Taking $\psi = \rho$ above, we see that
\[
 \mu_\rho\big((r_0 + r_1 z + \dots + r_m z^m)z^n\big) = \mu_\rho(r_0 z^n)
\]
is also supported on a singleton. Therefore, using \cref{inv of novikov elements} again, we see that $q_iy_{ii} = (r_0 + r_1 z + \dots + r_m z^m)z^n$ is invertible in $\widehat{ \K H}^{C'''}$, where $C'''$ is a dual of a vertex of $P(q_i y_{ii})$ containing $\rho$. Now it is clear that $C'' \subseteq C'''$ since for every $\xi \in C''$ we have
\[
 \mu_\xi(q_iy_{ii}) = \mu_\xi(r_0 z^n) = \mu_\rho(r_0 z^n) = \mu_\rho(q_iy_{ii})
\]
Hence $q_iy_{ii}$, and therefore also $y_{ii}$, is invertible in $\widehat{ \K H}^{C''}$.
We set $s_i = r_0$, and again allow $i$ to vary.
We set $Q = P(p) + \sum P(s_i)$.

Let $C_0$ be the intersection of $C$ with some dual of a vertex of $Q$. A dual of a vertex of $Q$ is a subset of duals of vertices of the summands $P(p)$ and $P(s_i)$. Therefore, by the discussion above, the matrix $A\otimes \widehat {\K H}^{C_0}$ is right-invertible, as claimed.

\step{2} We now claim that $C \s- \Sigma(A)$ has empty interior and 
is \emph{$\phi$-convex}, that is for any $\psi \in C \s- \Sigma(A)$ we have $s \psi + t \phi \in C \s- \Sigma(A)$ for every $(s,t) \in (0,\infty) \times [0,\infty)$.

The first property is immediate, as $C \cap \Sigma(A)$ contains the intersection of $C$ with the union of all duals of vertices of $Q$ (by Step $1$). The union of all duals of vertices of any polytope is dense in $H^1(H;\R)$; this in particular applies to $Q$. Since $C$ is open, its intersection with a dense subset of $H^1(H;\R)$ is dense in $C$.

Now let us look at the second property. Since $Q$ is $\phi$-flat, every dual of a face of $Q$ is $\phi$-convex (this is easy to verify). Also, the intersection of such a dual with $C$ is still $\phi$-convex, as $C$ is $\phi$-convex (since it is convex and contains $\phi$), and the intersection of two $\phi$-convex sets is $\phi$-convex (this is immediate). Therefore, it is enough to show that $C \cap \Sigma(A)$, and therefore also $C \s- \Sigma(A)$, is the intersection of $C$ with some union of duals of faces of $Q$.

Let $\psi \in C\cap \Sigma(A)$ be any character, and let $F$ denote a face of $Q$ in whose dual $\psi$ lies. Let $C_1$ denote the connected component of $\psi$ in $C \cap \Sigma(A)$. By \cref{bns open}, there exists an open neighbourhood $U$ of $\psi$ in $\Sigma(A)$. It is easy to see that for every vertex $v$ of $F$, the neighbourhood $U$ intersects some dual $C_v$ of $v$ in a non-trivial manner. \cref{bns summary} implies that $C_1$ contains  $C \cap C_v$ for every $v$, as the sets $C_v$ are contained in $C\cap\Sigma(A)$ by Step $1$, and they are convex and thus connected; moreover, \cref{bns summary} says that $C_1$ is convex. Thus, $C_1$ contains the entire intersection of $C$ with the dual of $F$ containing $\psi$. This shows that every $C_1$ is the intersection of $C$ with  some union of duals of faces of $Q$, as claimed.

\step{3} We now claim that $C \subseteq \Sigma(A)$, which will finish the proof thanks to \cref{bns summary}, since $C$ is connected, and therefore is a subset of a connected component of $\Sigma(A)$.

Suppose for a contradiction that $C \not\subseteq \Sigma(A)$: let $\rho \in C \s- \Sigma(A)$ be a character. Since $C$ is a dual of a vertex, it contains an open neighbourhood $U$ of $\rho$; we find integral characters $\phi_1, \dots, \phi_k \in U$ which form a basis of the $\R$-vector space $H^1(H;\R)$, and such that $\rho$ cannot be generated (over $\R$) using fewer than $k$ of these characters. We span a non-degenerate $k$-simplex in $H_1(H;\R)$ with vertices $\rho, \phi_1, \dots, \phi_k$, 
which necessarily has non-empty interior. Now we use the fact that $C \s- \Sigma(A)$ is $\phi_i$-convex for every $i$ (shown in Step $2$, as $\phi$ was an arbitrary element of $C \cap H^1(H;\Z)$), and conclude that the interior of the simplex lies in $C \s- \Sigma(A)$. But this contradicts the fact that $C \s- \Sigma(A)$ has empty interior.
\end{proof}

\section{Agrarian groups}
\label{sec agrarian}

In this section we introduce agrarian groups, that is groups whose (untwisted) integral group rings embed in skew-fields (division algebras). Such groups will be of central importance for us -- the embedding into a skew-field allows us to take determinants of square matrices over the group ring; the determinants in turn lead to their Newton polytopes which in many cases control the Bieri--Neumann--Strebel invariants.

 To prove the inheritance properties and to list known classes of examples of agrarian groups we need to first cover the Atiyah conjecture.

\subsection{The Atiyah conjecture}

Recall our important convention: whenever we talk about a matrix, we will explicitly state over which ring it lies. 

Let $A$ be a matrix over the integral group ring $\Z G$ of a group $G$. Naturally, $A$ represents a linear map between two finitely generated free right $\Z G$ modules, say
\[
 A \colon \Z G^n \to \Z G ^m
\]
Tensoring with $L^2(G)$ we obtain
\[
 A \otimes L^2(G) \colon L^2(G)^n \to L^2(G)^m
\]
Let $\mathcal N(G)$ denote the  \emph{von Neumann algebra} of $G$, that is the algebra of bounded $G$-equivariant operators on $L^2(G)$.
The kernel of $A \otimes L^2(G)$ is a \emph{Hilbert $\mathcal{N}(G)$-module}, which essentially means that it is a topological space with a continuous module structure over $\mathcal N(G)$.

The von Neumann algebra has two properties of interest here. Firstly, the non-zero-divisors in $\mathcal N(G)$ satisfy the Ore condition, and so $\mathcal N(G)$ embeds into its Ore localisation. Secondly, Hilbert $\mathcal{N}(G)$-modules have a \emph{von Neumann dimension}, which is an element of $[0,\infty]$ -- for details see the book of L\"uck~\cite[Section 1.1]{Lueck2002}).

\begin{dfn}
 Let $G$ be a torsion free group. We say that $G$ satisfies the \emph{Atiyah conjecture} \iff for every matrix $A$ over $\Z G$, the von Neumann dimension of $\ker A \otimes L^2(G)$ is an integer.
\end{dfn}

The Atiyah conjecture can also be formulated for groups with torsion -- for such a formulation, as well as a discussion of counterexamples, see \cite[Chapter 10]{Lueck2002}. There are however no torsion-free groups known which do not satisfy the Atiyah conjecture.

Also, one can consider the Atiyah conjecture over group rings of $G$ with other coefficients. This is however not relevant to our discussion here.

\begin{thm}[Linnell~{\cite{Linnell1993}}; Schick~{\cite{Schick2002}}; Linnell--Okun--Schick~{\cite{Linnelletal2012}}; Schreve {\cite{Schreve2014}}]
\label{atiyah groups}
Let $\mathcal C_1$ be the smallest class of groups containing all free groups which is closed under directed unions and extension by elementary amenable groups.

Let $\mathcal C_2$ be the smallest class of groups containing the trivial group, closed under subgroups, colimits and inverse limits over directed systems, and such that if we have a quotient $q \colon G \to H$ with $G$ torsion-free, $H$ elementary amenable and $q^{-1}(F) \in \mathcal C_2$ for every finite subgroup $F$ of $H$, then $G \in \mathcal C_2$.

Let $\mathcal C_3$ denote the union of the class of \{right-angled Artin\}-by-\{elementary amenable\} groups and the class of \{right-angled Coxeter\}-by-\{elementary amenable\} groups where in the latter case in addition the finite subgroups of the elementary amenable quotients are $2$-groups.

Let $\mathcal C_4$ denote the class of virtually cocompact special groups.

If $G$ is a torsion-free group lying in $\mathcal C_1 \cup \mathcal C_2\cup \mathcal C_3 \cup \mathcal C_4$, then $G$ satisfies the Atiyah conjecture.
\end{thm}

Note that $\mathcal C_2$ contains all residually \{torsion-free solvable\} groups, which is a very rich class. Also, let us elaborate on the class $\mathcal C_4$: it contains random groups (in the density model) with density less than $1/6$ -- this follows from the result of Olliver--Wise~\cite{OlliverWise2011}, who have shown that such random groups are cocompactly cubulated, combined with the fact that they are hyperbolic, and cocompactly cubulated hyperbolic groups are virtually cocompact special by the work of Agol~\cite[Theorem 1.1]{Agol2013}.


\begin{lem}[{\cite[Lemma 10.4]{Lueck2002}}]
\label{atiyah inherit}
The class of torsion-free groups satisfying the Atiyah conjecture is closed under taking subgroups and directed unions.
\end{lem}

In fact one can also give statements for groups with torsion, but they will not be relevant in our context.

We did not give a proper definition of the von Neumann dimension appearing in the Atiyah conjecture, since we will be using a reformulation due to Linnell. To state it we need one more definition.

\begin{dfn}[Division closure]
Let $R$ be a subring of a ring $S$. The \emph{division closure} of $R$ in $S$ is the smallest subring of $S$ containing $R$ such that if an element of this subring is invertible in $S$, then it is also invertible in the subring.
\end{dfn}

\begin{thm}[Linnell~{\cite{Linnell1993}}]
\label{atiyah field}
Let $G$ be a torsion-free group, and
 let $\D(G)$ denote the division closure of  $\Z G$ in the Ore localisation of the von Neumann algebra $\mathcal N(G)$. The group $G$ satisfies the Atiyah conjecture \iff $\D(G)$ is a skew-field.
\end{thm}

Note that in Linnell's paper, $\D(G)$ denotes the division closure of $\C G$, rather than $\Z G$. The proofs however work verbatim for $\Z G$.



\begin{prop}
\label{props of D}
Let $G$ be a torsion-free group satisfying the Atiyah conjecture. Every automorphism of $G$ extends to an automorphism of $\D(G)$. Also, if $K$ is a subgroup of $G$ then the natural embedding $\Z K \into \Z G$ extends to an embedding $\D(K) \to \D(G)$.
\end{prop}
\begin{proof}
Let $x$ be an automorphism of $G$. It is easy to see that $x$ extends to an automorphism of $L^2(G)$, hence of $\mathcal N(G)$, and therefore of the Ore localisation of $\mathcal N(G)$.

The second statement follows immediately from observing that the von Neumann algebra of a subgroup embeds into the von Neumann algebra of a supergroup in such a way that non-zero-divisors stay non-zero-divisors.
\end{proof}

The above proposition allows us in particular to construct a twisted group ring $\D(K) G/K$, when $K$ is normal.

\begin{prop}[{\cite[Lemma 10.69]{Lueck2002}}]
\label{linnell ore}
 If $G$  is a torsion-free group satisfying the Atiyah conjecture and if $G$ fits into a short exact sequence
 \[
  K \to G \to H
 \]
where $H$ is finitely-generated and abelian, then the Linnell skew-field $\D(G)$ is isomorphic to the Ore localisation of the twisted group ring $\D(K) H$.
\end{prop}

Since we are talking about the Atiyah conjecture, let us introduce the $L^2$-homology.

\begin{dfn}
The $n^{th}$ $L^2$-homology group of a group $G$ is defined to be \[H_n(G;L^2(G))\]
with the caveat that one computes the homology topologically, that is one divides the cycles by the closure of the boundaries.

A group $G$ is \emph{$L^2$-acyclic} \iff its $L^2$-homology vanishes for every $n$. In this context one may in fact take the usual notion of homology, without taking closures.
\end{dfn}

\begin{prop}[L\"uck~{\cite[Lemma 10.28(3)]{Lueck2002}}]
\label{l2 acyclic via D}
Let $G$ be a torsion-free group satisfying the Atiyah conjecture. Then \[\dim_{\mathcal N (G)} H_n(G;L^2(G)) = \dim_{\D(G)}H_n(G;\D(G))\] for every $n$.
\end{prop}

Again, L\"uck deals with the division closure of $\C G$, but the proof carries over to the case of $\D(G)$ being the division closure of $\Z G$.

\subsection{Agrarian and equivariantly agrarian groups}
We now introduce the class of groups of main interest to us.

\begin{dfn}
 A group $G$ is \emph{agrarian} \iff the integral group ring $\Z G$ embeds into a skew-field. If the embedding can be made equivariant with respect to $\Aut(G)$ (which also acts on $\Z G$), then $G$ is said to be \emph{equivariantly agrarian}.
\end{dfn}

If $\Z G$ embeds in a skew-field $\D$, then we treat $\D$ as a left $\Z G$ module, and a right $\D$ module. Thus we will tensor $\Z G$ modules on the right with $\D$, and obtain right $\D$ modules.

Note that $\Z G$ embeds in a skew-field $\D$ \iff  $\Q G$ embeds in $\D$.
\smallskip

Let us list some properties of agrarian groups.

\begin{prop}
\begin{enumerate}
 \item Every equivariantly agrarian group is agrarian.
 \item Subgroups of agrarian groups are agrarian.
 \item Countable directed unions of agrarian groups are agrarian. \label{agrarian unions}
 \item Agrarian groups are torsion free and satisfy the zero-divisor conjecture of Kaplansky, that is their integral group rings have no zero-divisors.
 \item Torsion-free groups satisfying the Atiyah conjecture are equivariantly agrarian.
 \item Amenable groups whose integral group rings have no zero-divisors are equivariantly agrarian.
 \item Biorderable groups are agrarian.
 \item \{Equivariantly agrarian\}-by-$H$ groups are agrarian, provided that $H$ is biorderable or $H$ is amenable and its twisted group rings with skew-field coefficients have no zero-divisors.
 \item Countable fully-residually agrarian groups are agrarian.
\end{enumerate}
\end{prop}
\begin{proof}
 \begin{enumerate}
  \item Obvious.
  \item Obvious.
  \item Let $(G_i)$ be a sequence of agrarian groups, and suppose that $j_i \colon \Z G_i \to \D_i$ is an embedding into a skew-field for every $i$. Let $\D$ denote the ultraproduct of the skew-fields $\D_i$; it is a standard exercise to show that $\D$ is itself a skew-field. Now we define $j \colon \Z (\bigcup G_i) \to \D$ by putting $j(g) = (j_i(g))$ for every $g \in \bigcup G_i$ (where we put $j_i(g) = 0$ if $g \not \in G_i$), and then extending linearly. Take $x \in \Z (\bigcup G_i)$: it is clear that $j(x) = 0$ implies that $j_i(x) = 0$ for infinitely many values $i$, and hence that $x = 0$. Thus $j$ is injective.
  \item Skew-fields have no zero-divisors, and groups with torsion have zero-divisors in their integral group rings.
  \item Follows from \cref{atiyah field}.
  \item Note that if $\Z G$ has no zero divisors, then $\Q G$ also does not admit any. Therefore, we may apply
  \cref{tamari} and \cref{Ore loc}(2),(4) to embed $\Q G$, and hence $\Z G$, into a skew-field in an equivariant way.
  \item Follows from \cref{malcev-neumann}.
  \item Let $K \to G \to H$ be an extension, with $K$ equivariantly agrarian. The very definition tells us that the untwisted group ring $\Z K$ can be embedded into a skew-field $\K$ in an $\Aut(K)$-equivariant fashion. In particular, the embedding is  equivariant with respect to the conjugation action of $G$ on $K$, and so we have
  \[
   \Z G = (\Z K) H \into \K H
  \]
Now the statement follows from \cref{malcev-neumann} when $H$ is biorderable and from \cref{tamari} when $H$ is amenable and $\K H$ has no zero-divisors.
\item This is similar to \cref{agrarian unions}: if $G$ is fully-residually agrarian, then we obtain a sequence $q_i$ of epimorphisms $q_i \colon G \to H_i$ such that every $H_i$ is agrarian, and for every $x \in \Z G$ there exists $j_0$ such that $q_j(x) \neq 0$ for every $j \geqslant j_0$, where $q_j$ now denotes the induced map on group rings. We embed $\Z H_i$ into a skew-field $\D_i$, and then take $\D$ to be the ultraproduct of the skew-fields. The maps $q_j$ define an embedding $\Z G \into \D$.
\qedhere
 \end{enumerate}
\end{proof}

We see that agrarian groups form a large class -- they contain torsion-free groups satisfying the Atiyah conjecture (including all elementary amenable groups and free groups), but they also contain extensions of such groups by free groups (as free groups are biorderable). Thus, for example, free-by-free groups are agrarian, but are not all known to satisfy the Atiyah conjecture; a related class of hyperbolic extensions of free groups has been studied by Dowdall--Taylor~\cite{DowdallTaylor2018}. Another example of agrarian groups which are not known to satisfy the Atiyah conjecture is the class of fundamental groups of $4$-manifolds which up to homotopy are surface-by-surface bundles (studied e.g. by Hillman~\cite{Hillman1991}) -- again, these are agrarian since they are Atiyah-by-biorderable.

 There is no torsion-free example of a non-agrarian group known.


\begin{lem}
\label{agrarian embedding}
 Let $G$ be an agrarian group, and suppose that $K \to G \to H$ is an exact sequence of groups. Then $\Z K$ embeds into a skew-field $\K$ in a $G$-equivariant way, and so we have an embedding
 \[
  \Z G \cong (\Z K) H \into \K H
 \]
\end{lem}
\begin{proof}
 Let $\K$ be the skew-field coming from the definition of agrarianism of $G$. Since $G$ embeds therein, we have an action of $G$ on $\K$ by conjugation. It is clear that this action preserves $\Z K$.
\end{proof}

Note that if we know that $G$ is torsion free and satisfies the Atiyah conjecture, then we may take $\K = \D(K)$ above, and then the Ore localisation of $\D(K) H$ is isomorphic to $\D(G)$ by \cref{linnell ore}.

\section{Applications}
\label{sec apps}

Throughout the applications $G$ is a group, and we  set $H = \fab{G}$, the free part of the abelianisation of $G$, $K = \ker( G \to \fab{G})$, and identify the untwisted group ring $\Z G$ with the twisted group ring $(\Z K) H$, as in \cref{key example}.

If we know that $G$ is agrarian (which we typically will), we will embed $\Z K$ into a skew-field $\K$, as explained in \cref{agrarian embedding}, and hence obtain an embedding
\[
 \Z G \cong (\Z K) \into \K H
\]
We will use $\D$ to denote the Ore localisation of $\K H$.

Tensoring in this section is always taken over $\Z G$.

\subsection{Bieri--Neumann--Strebel invariants}

In \cite{Bierietal1987} Bieri--Neumann--Strebel introduced the geometric invariant $\Sigma(G)$ (also known as the $\Sigma$-invariant or the BNS-invariant), where $G$ is any finitely-generated group. Later, Bieri--Renz~\cite{BieriRenz1988} extended the definition to a series of invariants $\Sigma^m(G;\Z)$, with $\Sigma(G) = \Sigma^1(G;\Z)$. We recall the definition below.
(Note that for us $\Z$ is a right module, and so we avoid the minus sign which usually appears in the equality above.)

\begin{dfn}
For every $m \in \N \sqcup \{\infty\}$, we define \[\Sigma^m(G;\Z)\subseteq H^1(G;\R) \s- \{0\}\] by declaring  $\phi \in \Sigma^m(G;\Z)$ \iff the monoid $\{ g \in G \mid \phi(g) \geqslant 0 \}$ is of type $\typeFP{m}$.
We will write $\Sigma^1(G)$ for $\Sigma^1(G;\Z)$.

Recall that a monoid $M$ is of type $\typeFP{m}$ \iff the trivial $\Z M$-module $\Z$ admits a resolution by projective $\Z M$-modules such that the first $m$ terms of the resolution are finitely generated (here $\Z M$ is the untwisted monoid ring). When $m=\infty$ we require all the terms of the resolution to be finitely generated.
\end{dfn}

Note that $M$ is of type $\typeFP{\infty}$ \iff it is of type $\typeFP{m}$ for all $m \in \N$, and so $\Sigma^\infty(G;\Z) = \bigcap_{m\in \N} \Sigma^m(G;\Z)$.

The definition of type $\typeFP{m}$ applies of course also to groups. In the context of groups let us introduce one more notion.

\begin{dfn}
 A group $G$ is \emph{of type $\typeF{}$} \iff it admits a finite CW-model of its classifying space, and \emph{of type $\typeF{\infty}$} \iff it admits a CW-model of its classifying space with finite skeleton in every dimension.
\end{dfn}

\begin{rmk}
 Note that Bieri--Renz talk about $\Sigma^m(G;\Z)$ as a subset of the character sphere obtained by identifying classes in $H^1(G;\R) \s- \{0\}$ under positive homotheties. We prefer to define $\Sigma^m(G;\Z)$ as a subset of the whole cohomology -- in our setup $\Sigma^m(G;\Z)$ is of course stable under multiplication by positive scalars.
\end{rmk}

\begin{thm}[Bieri--Renz~{\cite[Theorem A]{BieriRenz1988}}]
\label{bns open orig}
Let $G$ be a finitely generated group, and let $m \in \N$. The sets $\Sigma^m(G;\Z)$ are open subsets of $H^1(G;\R)$.
\end{thm}

\subsection{Sikorav's theorem}

Fundamental to the way we will treat BNS-invariants is the following theorem of Sikorav~\cite{Sikorav1987}, which gives an interpretation of $\Sigma^m(G;\Z)$ in terms of group homology of $G$ with coefficients in the Novikov ring $\widehat {\Z G}^\phi$. (We give a reference to an easily available paper of Bieri, rather than to the original thesis of Sikorav.)

\begin{thm}[Sikorav {\cite[Theorem 2]{Bieri2007}}]
\label{sikorav}
 Let $G$ be a group of type $\typeFP{m}$. For every $\phi \in H^1(G;\R) \s- \{0\}$ we have $\phi \in \Sigma^m(G;\Z)$ \iff $H_i(G;\widehat {\Z G}^\phi) = 0$ for every $i \leqslant m$.
\end{thm}

In practice, we will compute $H_i(G;\widehat {\Z G}^\phi)$ using the cellular chain complex of the universal covering of some CW-model for the classifying space of $G$; the chain complex is naturally a free resolution of $\Z$ by $\Z G$-modules. The group homology with coefficients in the Novikov completion $\widehat {\Z G}^\phi$ is obtained by tensoring the chain complex with $\widehat {\Z G}^\phi$ and then computing the homology.

When the group $G$ is of type $\typeF{\infty}$, then we can arrange for the above complex to be a complex of free finitely-generated $\Z G$-modules; hence the boundary maps are naturally matrices over $\Z G$, and the homology computation is related to their invertibility upon extension of scalars.
This is the reason for our interest in $\Sigma$-invariants of matrices.

\subsection{\texorpdfstring{$\Sigma$}{Sigma}-invariants of matrices revisited}

Let us record the following, which highlights some similarities between BNS-invariants for groups and for matrices over a group ring.

\begin{lem}
Let $G$ be a group, and let $A$ be a square matrix over $\Z G$.
 Then $\Sigma(A)$ is stable under positive homotheties, and is an open subset of $H^1(G;\R)$.
\end{lem}
\begin{proof}
 The first assertion is obvious, since multiplying $\phi$ by a positive scalar does not change the ring $\widehat{\Z G}^\phi$.

The latter assertion follows immediately from \cref{entries far away}.
\end{proof}

To discuss the $\Sigma$-invariants of matrices further, we introduce the notion of a marked polytope.

\begin{dfn}
A function $m \colon H^1(G;\R) \s- \{0\} \to \{0,1\}$ (where $\{0,1\}$ is endowed with the discrete topology) is a \emph{marking function} \iff $m^{-1}(0)$ is closed. A character mapped to $1$ will be called \emph{marked}. 

Let $P$ be a polytope. We say that $m$ is a \emph{marking of $P$} \iff $m$ is constant on duals of faces of $P$. A dual mapped to $1$ will be called \emph{marked}. A polytope endowed with a marking will be called a \emph{marked polytope}.
When $m^{-1}(1)$ consists solely of duals of vertices, we say that $P$ is a polytope with \emph{marked vertices}.
\end{dfn}

Note that the definition above leads to statements which may seem a little awkward at first, since we will often assume the existence of a marked polytope, and then immediately treat being marked as a property of a character.

\begin{rmk}
When every face of $P$ admits a single dual, the definition of a marked polytope above gives the more natural notion of a marking, where it is faces of the polytope which are marked. The fact that $m^{-1}(0)$ is closed guarantees that faces of marked faces are also marked.
\end{rmk}

\smallskip

Recall that when $G$ is agrarian, we have a skew-field $\K$ containing $\Z K$, where $K = \ker ( G \to \fab{G})$, and hence we also have the Ore localisation $\D$ of the twisted group ring $\K \fab{G}$.

\begin{thm}
\label{bns matrix}
Suppose that $G$ is agrarian.
Let $A$ be a square matrix over $\Z G$ with $A \otimes \D$ invertible. The polytope $P(A)$ admits a marking of its vertices, such that for every ${\phi \in H^1(H;\R) \s- \{0\}}$ we have $\phi \in \Sigma(A)$ \iff $\phi$ is marked.
\end{thm}
\begin{proof}
This is a fairly straightforward corollary of \cref{K-bns matrix}: we know that $A\otimes \widehat {\K H}^\phi$ admits a right-inverse \iff $\phi$ lies in a dual of a vertex of $P(A)$. Now being right-invertible over $\widehat {\K H}^\phi$ is certainly a necessary condition for being right-invertible over $\widehat {\Z G}^\phi$, as $\widehat {\Z G}^\phi$ is a subring of $\widehat {\K H}^\phi$. Thus characters lying in duals of faces of dimension at least $1$ do not lie in $\Sigma(A)$.

We now define the marking on $P(A)$. A dual $C$ of a vertex is marked \iff $C$ intersects $\Sigma(A)$ non-trivially. We need only show that if $C$ intersects $\Sigma(A)$ non-trivially, then in fact $C \subseteq \Sigma(A)$. Take $\phi, \psi \in C$, with $\phi \in \Sigma(A)$. \cref{K-bns matrix} tells us that $A\otimes \widehat {\K H}^C$ is right-invertible; let $X$ denote its inverse. We also know that $A$ admits a right inverse $Y$ over $\widehat {\Z G}^\phi$. Right-inverses over $\widehat {\K H}^\phi$ are unique by \cref{unique RI}, and so $X=Y$, which in turn implies that $X$ lies over $\widehat {\K H}^C \cap \widehat {\Z G}^\phi=\widehat {\Z G}^C$, and we are done.
\end{proof}

Let us now discuss ways of combining  various marked polytopes into a single one.

\begin{dfn}
An \emph{atlas of markings} consists of a finite open cover $U_1, \dots, U_n$ of $H^1(G;\R) \s- \{0\}$, and a finite collection $(P_1,m_1), \dots, (P_n,m_n)$ of marked polytopes in $H_1(G;\R)$, such that
 for every pair $(i,j)$ we have $m_i\vert_{U_i \cap U_j} = m_j\vert_{U_i \cap U_j}$.
\end{dfn}

It is clear that an atlas of markings gives a well defined marking function $m\colon H^1(G;\R) \s- \{0\} \to \{0,1\}$. Hence, we will often specify $m$ instead of the markings $m_i$. It is clear that $m^{-1}(1)$ is open, since it is open inside every chart $U_i$. Thus $m^{-1}(0)$ is closed.

\begin{lem}
\label{one marked polytope}
Let $U_1, \dots, U_n$ and $P_1, \dots, P_n$ be an atlas of markings. The induced marking function $m$ defines a marking of
$P = \sum_{i=1}^n P_i$.
\end{lem}
\begin{proof}
Recall \cref{faces decomp}: for any face $Q$ of $P$, we have unique faces $Q_i$ of $P_i$ such that $Q = \sum_{i=1}^n Q_i$.

Let $C$ denote a dual of a face $Q$ of $P$. We argue that either all characters in $C $ are marked, or none is. To this end, let $\phi \in C$.
There exists $i$ such that $\phi \in U_i$, and $\phi$ is marked \iff the dual of $F_\phi(P_i)$ containing $\phi$ is marked. Now,
for every $\psi$ in the connected component of $C \cap U_i$ containing $\phi$ we have $F_\phi(P_i) = Q_i = F_\psi(P_i)$, and so $\psi$ is marked \iff $\phi$ is. This shows that the marking map $m \colon \bigcup U_i \to \{0,1\}$ is continuous on $C$, which is connected. Since the target of the marking map is discrete, the map must be constant on $C$.
%
\end{proof}

\begin{lem}
\label{one marked L2 polytope}
Let $\S = \{s_1, \dots, s_n\}$ be a finite generating set of $G$, let $P$ be a polytope in $H_1(G;\R)$, and let $k_i \in \N \cup \{0\}$ be given for every $i \in \{ 1, \dots, n\}$.
Suppose that for every $i$ there exists a marking of vertices of $P_i = P + k_i \cdot P(1-s_i)$ such that the collections $U_1, \dots, U_n$ and $P_1, \dots, P_n$ form an atlas of markings, where \[U_i = \{ \phi \in H^1(G;\R) \mid \phi(s_i) \neq 0 \}\]
 Then the induced marking function $m$ defines a marking of vertices of $P$.
\end{lem}
\begin{proof}
The proof is identical to the previous one, with only one exception: the proof that $F_\phi(P_i) = F_\psi(P_i)$ holds for every $\psi$ in the connected component of $C \cap U_i$ containing $\phi$, where $C$ is a dual, is different. To show the desired statement, take such $C,U_i, \phi$ and $\psi$. Since $\phi$ and $\psi$ lie in a connected component of $U_i$, the real numbers $\phi(s_i)$ and $\psi(s_i)$ have the same sign. Therefore
\[F_\phi\big(k_i P(1-s_i)\big) = F_\psi\big(k_i P(1-s_i)\big)\]
Also, $\phi, \psi \in C$ implies that $F_\phi(P) = F_\psi(P)$. Thus $F_\phi(P_i) = F_\psi(P_i)$ as required.
%
\end{proof}

\subsection{Polytope class}

In \cite{Funke2018} Funke introduced the following notion.
\begin{dfn}
 Let $G$ be a torsion-free group satisfying the Atiyah conjecture and such that the abelianisation of $G$ is finitely generated. We say that $G$ is of \emph{polytope class} \iff for every square matrix $A$ over $\Z G$ the following holds: if $A \otimes \D(G)$ is invertible then the Newton polytope $P(\det A \otimes \D(G))$ is a single polytope.
\end{dfn}

Recall that $\D(G)$ is the skew-field of Linnell -- see \cref{atiyah field}. 

Funke proved in ~\cite[Theorem 4.1]{Funke2018} that torsion-free amenable groups $G$ satisfying the Atiyah conjecture and with finitely generated abelianisation are of polytope class. We show that one does not need to assume amenability.

\begin{thm}
\label{polytope class}
Every torsion free group $G$ with finitely generated abelianisation, and with $G$ satisfying the Atiyah conjecture, is of polytope class.
\end{thm}
\begin{proof}
Recall that in this context we have the Linnell skew-field $\D(G)$ containing $\Z G$. Moreover, $\D(G)$ is isomorphic to the Ore localisation of the twisted group ring $\D(K) \fab{G}$, where $K = \ker (G \to \fab{G})$, by \cref{linnell ore}. Thus, we take $\K = \D(K)$ and $\D = \D(G)$, and we apply \cref{single poly}.
\end{proof}

This fact has interesting consequences for the Whitehead group.

\begin{dfn}
Let $\GL(\Z G)$ denote the directed union over $n$ of the groups of invertible $n \times n$ matrices over $\Z G$, where the embeddings are given by introducing the identity matrix in the bottom-right corner. The group $K_1(G)$ is the abelianisation of $\GL(\Z G)$. The \emph{Whitehead group} $\mathrm{Wh}(G)$ is the quotient of $K_1(G)$ by the subgroup generated by $1 \times 1$ matrices of the form $(\pm g)$, with $g \in G$.
\end{dfn}

Suppose now that $G$ is agrarian, and let $\D$ be a skew-field as discussed at the beginning of \cref{sec apps}.
It is not hard to see that the map $A \mapsto \det A \otimes \D$ is well-defined on $\GL(\Z G)$. Hence, one can talk about the Newton polytopes of elements of $\GL(\Z G)$. It is even easier to see that if we quotient the polytope group $\P(\fab{G})$ by the translations (and obtain $\P_T(\fab{G})$), then in fact we obtain a well defined map $\mathrm{Wh}(G) \to \P_T(\fab{G})$. Thus, we may talk about the Newton polytopes of elements of the Whitehead group.

It is an open question whether there exists a torsion-free group $G$ with non-trivial $\mathrm{Wh}(G)$. We show that in our setting at least the Newton polytopes vanish.

\begin{cor}
\label{Whitehead}
 Let $G$ be an agrarian group  with finitely generated abelianisation. The Newton polytope of every element of $\mathrm{Wh}(G)$ is a singleton.
\end{cor}
\begin{proof}
Since $G$ is agrarian, fix a skew-field $\D$ into which $\Z G$ embeds.

 Let $A \in \mathrm{Wh}(G)$ be given. We can think of $A$ as a square matrix over $\Z G$, which is invertible.
Recall that $P(A)$ denotes $P(\det A \otimes \D)$.

We have
 \[
  P(A) + P(A^{-1}) = P(AA^{-1}) = P (\det \I \otimes \D) = 0
 \]
and so $P(A) = -P(A^{-1})$. But both are single polytopes, and so $P(A) = P(A^{-1}) = 0$ in $\P_T(G)$.
\end{proof}

Note that a version of the above corollary for torsion-free groups satisfying the Atiyah conjecture is proven in \cite[Lemma 3.4]{Funke2018} under the extra assumption that $G$ is of $P\leq0$-class, which is implied by being of polytope class.

The corollary has a direct application to the study of the $L^2$-torsion polytopes.
These polytopes are defined for any free finite  $G$-CW-complex $X$ which is $L^2$-acyclic, provided that $G$ satisfies the Atiyah conjecture. The construction is due to Friedl--L\"uck~\cite{FriedlLueck2017}, and is a little complicated -- roughly speaking, to $X$ one associates an $L^2$-version of the Whitehead torsion, called the universal $L^2$-torsion. This torsion is naturally an element of $\D(G)^\times / [\D(G)^\times, \D(G)^\times]$, and hence one can take its Newton polytope up to translation. This element of $\P_T(\fab{G})$ is precisely the $L^2$-torsion polytope $P_{L^2}(X,G)$.

The universal $L^2$-torsion is a simple homotopy invariant, but not a homotopy invariant -- when one passes to a complex $G$-homotopy equivalent to $X$, the universal $L^2$-torsion gets multiplied by an element of $\mathrm{Wh}(G)$. When passing to the $L^2$-torsion polytope, the situation was a priori the same. However, the corollary above tells us that the Newton polytopes of elements in $\mathrm{Wh}(G)$ are trivial, and so the $L^2$-torsion polytope is actually a homotopy invariant.
Let us summarise this discussion.

\begin{thm}
Let $G$ be an $L^2$-acyclic group of type $\typeF{}$, and suppose that $G$ satisfies the Atiyah conjecture. For any two finite models $X$ and $Y$  for the classifying space of $G$, we have
\[
 P_{L^2}(X,G) = P_{L^2}(Y,G)
\]
Therefore, we may purge $X$ from the notation and talk about $P_{L^2}(G)$, the $L^2$-torsion polytope of $G$.
\end{thm}

\subsection{Vanishing of the \texorpdfstring{$L^2$}{L\texttwosuperior}-torsion polytope in the presence of amenability}
\label{sec: vanishing}

We now use \cite[Lemma 5.2]{Funke2018} and the proof of \cite[Theorem 5.3]{Funke2018} of Funke to prove the following.

\begin{thm}
\label{amenable}
Let $G$ be a group of type $\typeF{}$ satisfying the Atiyah conjecture. Suppose that $G$ admits a finite chain of subnormal subgroups with the last term being a non-trivial amenable group $N$. Suppose further that $N$ is not abelian or that $G$ has trivial centre. Then $G$ is $L^2$-acyclic and the $L^2$-torsion polytope of $G$ is a singleton.
\end{thm}
\begin{proof}
Note that $N$ is amenable and infinite, since $G$ is of type $\typeF{}$, and thus torsion free. Therefore the $L^2$-acyclicity of $G$ follows from \cite[Theorem 1.44]{Lueck2002} of L\"uck. 

Since $G$ is of type $\typeF{}$, it is finitely generated, and hence so is its abelianisation. Thus $G$ is of polytope class by \cref{polytope class}.

Now, to fix notation, we denote the elements of the subnormal chain by $N_i$, with $N_0 = G$ and $N_k = N$.
Let $M = N \cap K$ (recall that $K = \ker G \to \fab{G}$).

\smallskip
\noindent \textbf{Case 1:} Suppose that $M$ is not trivial, and hence $M$ is an infinite amenable group.
Let $S$ be the set of non-zero elements of $\Z M$. Note that the untwisted group ring $\Z G$ has no zero-divisors, since it embeds into the skew-field $\D(G)$. Thus $\Z M$ has no zero-divisors either. Therefore the subset $S$ of $\Z M$ satisfies the Ore condition inside of $\Z M$.
Note that $M$ is a normal subgroup of $N_{k-1}$, and that $S$ is $N_{k-1}$-invariant. Now \cref{inheriting Ore} tells us that $S$ satisfies the left Ore condition in $\Z N_{k-1}$. Set $S_{k-1} = S$. Replacing $S_{k-1}$ by
$S_{k-2}$, the multiplicative closure of the $N_{k-2}$-conjugates of $S_{k-1}$, and using \cref{inheriting Ore} again, we see that $S_{k-2}$ satisfies the left Ore condition in $\Z N_{k-2}$. We repeat the procedure until we arrive at the closure $S_0$ which satisfies the left Ore condition in $\Z G$.
The same argument shows that $S_0$ satisfies the right Ore condition, and hence the Ore conditions on both sides.

Note that the Newton polytope of every element in $S_0$ is a singleton, since this is true for elements in $S$ by construction (as these lie in the group ring of the kernel $K$) and is clearly preserved under conjugation and multiplication.

We are now going to use \cite[Lemma 5.2]{Funke2018} of Funke.
Considering only a classifying space for the group in question, and using \cref{polytope class} to remove the assumption of $P \geq 0$-class (which is weaker than being of polytope class),
we may state it as follows:

\begin{quote}
 Let $G$ be a group of type $\typeF{}$ satisfying the Atiyah conjecture. Let $T \subseteq \Z G$ be a multiplicative subset satisfying the left and right Ore conditions, and such that $P(t)$ is a singleton for every $t \in T$. If $H_i(G;T^{-1} \Z G) = 0$ for every $i$, then $P_{L^2}(G)$ is a singleton.
\end{quote}

The assumptions of \cite[Lemma 5.2]{Funke2018} are satisfied (taking $T = S_0$) -- we need only 
that
\[H_n(G; S_0^{-1} \Z G) = 0\]
for all $n$. 
To prove it, note that extending scalars to a localisation is exact, and so the only potential problem occurs for  $H_0(G; S_0^{-1} \Z G)$, since  $H_i(G; \Z G) = 0$ for all $i > 0$, and $H_0(G; \Z G) = \Z$. But $S_0$ contains an element of the form $1-g$ with $g \in M \s- \{1\} \subseteq G \s- \{1\}$ by assumption, and so $H_0(G; S_0^{-1} \Z G) = 0$ (as we could compute it from a presentation complex coming from a presentation of $G$ containing $g$ as a generator, and then the first boundary map is clearly onto over $S_0^{-1} \Z G$).

\smallskip
\noindent \textbf{Case 2:} Suppose that $M$ is trivial. Then $N$ embeds into $H = \fab{G}$, and so is abelian.
Thus $G$ has trivial centre by hypothesis.

Since $N$ is normal in $N_{k-1}$, and embeds into $H = \fab{G}$, the action of $N_{k-1}$ on $N$ must be trivial, and so $N$ lies in the centre of $N_{k-1}$. We replace $N$ by the centre of $N_{k-1}$: if this new $N$ does not embed into $H$, we are back in case 1. Otherwise, note that $N$ is in fact characteristic in $N_{k-1}$, and so normal in $N_{k-2}$. We repeat the argument, and replace $N$ by the centre of $N_{k-2}$. We continue doing so until either the new $N$ intersects $K$ non-trivially, or until we have produced a non-trivial centre of $G$, which is impossible.
\end{proof}

The theorem above allows us to confirm a conjecture of Friedl--L\"uck--Tilmann.
\begin{conj}[Friedl--L\"uck--Tillmann~{\cite[Conjecture 6.4]{Friedletal2016}}]
Let $G$ be an $L^2$-acyclic group of type $\typeF{}$ satisfying the Atiyah conjecture and with $\mathrm{Wh}(G) = 0$. If $G$ is amenable but not virtually $\Z$, then $P_{L^2}(G)$ is a singleton.
\end{conj}

We prove the following, stronger version.

\begin{cor}
\label{FLT conj}
Let $G$ be an $L^2$-acyclic group of type $\typeF{}$ satisfying the Atiyah conjecture. If $G$ is amenable but not isomorphic to $\Z$, then $P_{L^2}(G)$ is a singleton.
\end{cor}
\begin{proof}
If $G$ is non-abelian, then the result follows directly from \cref{amenable} by taking $N = G$. If $G$ is abelian, then type $\typeF{}$ tells us that $G$ is a finitely generated free abelian group, which is not $\Z$. Therefore the universal $L^2$-torsion of $G$ vanishes by \cite[Theorem 2.5(5)]{FriedlLueck2017}. Hence $P_{L^2}(G) = 0$.
\end{proof}

Note also that $G=\Z$ fails the assertion of the conjecture: we have \[P_{L^2}(\Z) = - P(1-z)\] where $z$ is a generator of $\Z$.

\begin{rmk}
 \cref{amenable}  corresponds nicely with \cite[Theorem 4.10]{Wegner2009} of Wegner: we are assuming that $G$ satisfies the Atiyah conjecture, whereas Wegner assumed that $G$ has \emph{semi-integral determinant}, which is a related  condition. Also, in his theorem one requires $G$ to posses an elementary amenable subnormal subgroup, whereas for our purposes amenable subnormal subgroups suffice. Wegner proved the vanishing of the $L^2$-torsion, whereas we prove the vanishing of the $L^2$-torsion polytope. Both invariants are determined by the universal $L^2$-torsion of Friedl--L\"uck \cite{FriedlLueck2017}.
\end{rmk}

\subsection{Agrarian groups of deficiency \texorpdfstring{$1$}{1}}

Throughout this section, $G$ is going to be an agrarian group of deficiency $d \geqslant 1$, that is $G$ will admit a presentation with the number of generators exceeding the number of relators by $d$ (formally speaking, it also means that there is no presentation with an even greater discrepancy). We will fix a presentation of $G$ with a finite generating set $\S = \{s_1, \dots, s_k\}$ realising the deficiency. We will also build a CW-classifying space for $G$ whose $2$-skeleton coincides with the presentation complex, and let $(C_\ast, \partial_\ast)$ denote the cellular chain complex of the universal covering of the classifying space, with $\partial_n \colon C_n \to C_{n-1}$. We will choose representatives for the $G$-orbits of cells, and hence identify $C_\ast$ with a chain complex of free $\Z G$-modules. Note that, by our convention,  $C_\ast$ is a chain complex of right modules, and so the boundary maps are matrices acting on the left.

We will look at two boundary maps more closely. The first one, $\partial_1$, is the row vector $(1-s_i)_i$. The second one, $\partial_2$, is a matrix with $d$ more rows than columns, and we will denote it by $A$. We let $A_i$ denote $A$ with the $i^{th}$ row removed.


\begin{thm}[Bieri--Neumann--Strebel~{\cite[Theorem 7.2]{Bierietal1987}}]
 For any finitely generated group $G$, if $d>1$ then $\Sigma^1(G) = \emptyset$.
\end{thm}

From now on we will assume that $d=1$. Crucially for us, in this case the matrices $A_i$ are square, since $C_2$ has rank precisely one less than $C_1$.

Friedl conjectured that a group $G$ of deficiency one which admits an aspherical presentation complex realising the deficiency has $\Sigma^1(G)$ determined by a polytope. We prove this conjecture under the assumption of $G$ being agrarian, but we will not need the assumption of asphericity of the presentation complex.
\begin{thm}
\label{defic 1}
 Let $G$ be an agrarian group of deficiency $1$. Then there exists a marked integral polytope $P$ in $H_1(G;\R)$, 
 such that $\phi \in \Sigma^1(G)$ \iff $\phi$ is marked.
\end{thm}
\begin{proof}
%
In view of Sikorav's Theorem (\cref{sikorav}), for every character \[\phi \in H^1(G;\R) \s- \{0\}\] we have $\phi \in \Sigma^1(G)$ \iff
\[
 H_1(C_\ast \otimes \widehat{\Z G}^\phi) =  H_0(C_\ast \otimes \widehat{\Z G}^\phi) = 0
\]

Let $s_i \in \S$. If $\phi(s_i) \neq 0$, we can modify the complex $C_\ast\otimes \widehat {\Z G}^\phi$ without changing its homology as follows: let $v$ denote the $i^{th}$ basis vector of $C_1$. Since $\phi(s_i) \neq 0$, the element $1-s_i$ is invertible in $\widehat {\Z G}^\phi$, and so the boundary map $\partial_1 \otimes \widehat {\Z G}^\phi$ restricts to an isomorphism from the $\widehat {\Z G}^\phi$-span of $v$ onto $C_0 \otimes \widehat {\Z G}^\phi$. This immediately implies that $H_0(C_\ast \otimes \widehat{\Z G}^\phi) = 0$. Also, we may split off the span of $v$ from $C_1 \otimes  \widehat{\Z G}^\phi$ and write
\[
C_1 \otimes  \widehat{\Z G}^\phi = \langle v \rangle \oplus C_1'
\]
where $C_1'$ is spanned by the remaining basis vectors of $C_1$.
We will use this procedure repeatedly, whenever we are working over a ring $R$ in which $1-s_i$ is invertible. As a shorthand, we will say that we `replace  $C_1 \otimes R$ by the $R$-module $C_1'$'.

It is immediate that $H_1(C_\ast \otimes \widehat{\Z G}^\phi) = 0$ \iff the composition
\[
C_2\otimes \widehat{\Z G}^\phi \to C_1 \otimes \widehat{\Z G}^\phi \to C_1'
\]
of $\partial_2$ and the projection
is onto.
But this map is precisely $A_i \otimes \widehat{\Z G}^\phi$, and it is onto \iff $\phi \in \Sigma(A_i)$. This in turn is equivalent to $\phi$ being marked under $m_i$, where the marking $m_i$ comes from \cref{bns matrix} and is a marking of $P(A_i)$.

Now the collections $U_1, \dots, U_k$ and $P(A_1), \dots, P(A_k)$ form an atlas of markings, where \[U_i = \{ \psi \in H^1(G;\R) \s- \{0\} \mid \psi(s_i) \neq 0 \}\]
 An application of \cref{one marked polytope} allows us to combine the marked polytopes
 \[P(A_1), \dots, P(A_k)\]
 into a single marked polytope $P$ which satisfies the assertion of our theorem.
\end{proof}

\begin{rmk}
The bulk of the above proof will be used repeatedly in the sequel.
\end{rmk}

Under stronger hypotheses we can in fact prove a somewhat stronger theorem: we will assume now that $G$ satisfies the Atiyah conjecture, and use the $L^2$-torsion polytope $P_{L_2}(G)$.

\begin{thm}
\label{defic 1 atiyah}
 Let $G$ be a torsion-free group of deficiency $1$, which satisfies the Atiyah conjecture and has trivial first $L^2$-Betti number. 
 We have $\Sigma^1(G) = \Sigma^\infty(G;\Z)$, and there exists an integral polytope $P$ in $H_1(G;\R)$ with marked vertices,
 such that $\phi \in \Sigma^1(G)$ \iff $\phi$ is marked.

 Moreover, the Cayley $2$-complex of any presentation of $G$ realising the deficiency is contractible.
\end{thm}
\begin{proof}
We continue with the notation of the proof of \cref{defic 1}.
We first claim that the boundary map $\partial_3 \colon C_3 \to C_2$ is trivial.
Since $G$ satisfies the Atiyah conjecture and is torsion free, we have the skew-field $\D(G)$ of Linnell (introduced in \cref{atiyah field}) at our disposal. Since we also know that the first $L^2$-Betti number of $G$ vanishes, we have
$
 H_1(G;\D(G)) = 0
$
by \cref{l2 acyclic via D}.

Take $v$ to be the first basis vector of $C_1$. Since
$(1-s_1)$ is invertible over $\D(G)$, we argue precisely as in the  proof of \cref{defic 1}, and replace $C_1 \otimes \D(G)$ by $C_1'$, a free module over $\D(G)$ of rank one less than $C_1$. The vanishing of $H_1(G;\D(G))$ tells us that the matrix \[A_1\otimes \D(G) \colon C_2\otimes \D(G) \to C_1'\]
is onto. But the ranks of $C_2 \otimes \D(G)$ and $C_1'$ are equal, and so $A_1\otimes \D(G)$ is also injective. Thus $\partial_2$ is injective, and so $\partial_3 = 0$. This implies that the Cayley $2$-complex of $G$ is contractible, and so we may take $X$ to be the presentation complex (which is of dimension $2$).

Now we argue precisely as in the proof of \cref{defic 1}: letting $A = \partial_2$, and setting $A_i$ as before to be a square matrix over $\Z G$, we conclude that for every $\phi \in H^1(G;\R) \s- \{0\}$ we have $\phi \in \Sigma^1(G)$ \iff $\phi$ lies in the dual of a marked vertex of $P(A_i)$ for every $i$ with $\phi(s_i) \neq 0$. Now, instead of combining the polytopes $P(A_i)$ into a single polytope by taking the sum, we will use the fact that
\[
P(A_i) - P(1-s_i) = P_{L^2}(G)
\]
for every $i$ -- this follows from the additivity of the universal $L^2$-torsion that underpins the $L^2$-torsion polytope, see \cite[Lemma 1.9]{FriedlLueck2017}. (For the special case of descending HNN extensions of free groups, a more direct reference is \cite[Theorem 3.2(1)]{FunkeKielak2018}; the computation given there applies to the current setting as well.)

If the rank of $H= \fab{G}$ is $0$ then there is nothing to prove.
Suppose that the rank of $H$ is equal to $1$. We take $P = P(1-s_i)$, where $s_i$ is not trivial in $H$. We mark the vertices of $P$ in the unique way making the assertion about $\Sigma^1(G)$ true.

Now suppose that the rank of $H$ is at least $2$.
We claim that $P_{L^2}(G)$ is a single polytope. To this end, let $\phi \in H^1(G;\Z)$ be non-trivial. There exists $s_i \in \S$ such that $\phi(s_i) \neq 0$, and so
\[
F_\phi\big( P_{L^2}(G) \big) =  F_\phi\big( P(A_i) \big)
\]
up to translation, since $F_\phi\big( P(1-s_i) \big)$ is a singleton. But $P(A_i)$, and so $F_\phi\big( P(A_i) \big)$, is a single polytope.
Now the claim follows from \cref{funke}.

We conclude the existence of the desired marked polytope by applying \cref{one marked L2 polytope} with $P = P_{L^2}(G)$ and $P_i = P(A_i)$.

\smallskip

We are left with the claim about $\Sigma^i(G;\Z)$ for $i>0$. Let $\phi \in H^1(G;\R) \s- \{0\}$, and take $i$ such that $\phi(s_i) \neq 0$. We have shown above that $\phi \in \Sigma^1(G)$ \iff $A_i$ admits a right inverse over $\widehat { \Z G}^\phi$. But Bieri~\cite[Theorem 3]{Bieri2007} (building on the work of Kochloukova~\cite{Kochloukova2006}) showed that $\widehat { \Z G}^\phi$ is \emph{von Neumann finite}, which means that a square matrix over $\widehat { \Z G}^\phi$ admits a right-inverse \iff it admits a two-sided inverse. Thus $\phi \in \Sigma^1(G)$ \iff $A_i \otimes \widehat { \Z G}^\phi$ is an isomorphism. Also, $\phi \in \Sigma^2(G;\Z)$ \iff $\phi \in \Sigma^1(G)$ and $A_i \otimes \widehat { \Z G}^\phi$ is injective, by \cref{sikorav}. Thus we have established $\Sigma^1(G) = \Sigma^2(G;\Z)$. For $i>2$ we have $H_i(G;\widehat {\Z G}^\phi) = 0$ since $X$ is of dimension $2$, and so another application of \cref{sikorav} completes the proof.
\end{proof}

\subsection{Descending HNN extensions of free groups}

We are now going to focus on a special kind of groups of deficiency $1$, namely descending HNN extensions of finitely generated free groups.

\begin{dfn}
Let $F_n = \langle a_1, \dots, a_n \rangle$ denote the free group of rank $n$, and let $g \colon F_n \to F_n$ be a monomorphism. The \emph{descending HNN extension} $F_n \ast_g$ of $F_n$ by $g$ is defined by the presentation
\[
F_n \ast_g = \langle a_1, \dots, a_n, t \mid t^{-1} a_1 t g(a_1)^{-1}, \dots, t^{-1} a_n t g(a_n)^{-1} \rangle
\]
\end{dfn}
In particular, \{finitely generated free\}-by-$\Z$ groups are descending HNN extensions of free groups, where $g$ is an automorphism.

Note that changing the sign of the stable letter $t$ in the presentation gives the definition of an \emph{ascending} HNN extension.

Let us mention the following conjecture of Bieri:
\begin{conj}[Bieri~{\cite[Conjecture]{Bieri2007}}]
\label{bieris conj}
Let $G$ be a group of deficiency $1$ with $\Sigma^1(G) \neq \emptyset$. Then $G$ is a descending HNN extenson of a finitely generated free group.
\end{conj}

The conjecture is known to hold (by \cite[Corollary B]{Bieri2007}) if there exists a character $\phi \in H^1(G;\R) \s- \{0\}$ such that $\{\phi,-\phi\} \subseteq \Sigma^1(G)$.

Before proceeding to the main statement of this section, let us discuss semi-norms induced by polytopes, which will be important in the proof.

\begin{dfn}[Semi-norms]
Let $P$ be a polytope in $H_1(G;\R)$, where $G$ is a group with finitely generated abelianisation. We define
$
 \| \cdot \|_P \colon H^1(G;\R) \to [0,\infty)
$
by
\[
 \| \phi \|_P = \max_{x,y \in P} |\phi(x) - \phi(y)|
\]
It is easy to see that $ \| \cdot \|_P$ is a semi-norm.

Let $P-Q$ represent an element $R$ of $\P_T(G)$, where $P$ and $Q$ are single polytopes. We define \[\| \cdot \|_R = \| \cdot \|_P - \| \cdot \|_Q \colon H^1(G;\R) \to \R\]
Again, one readily sees that this function is independent of the choice of $P$ and $Q$.
\end{dfn}

When $P = P_{L^2}(G)$ is defined, it is interesting to study $\| \cdot \|_P$. This has in particular been done for descending HNN extensions of non-trivial free groups in \cite{FunkeKielak2018}, where it was shown that the function is indeed a semi-norm (see \cite[Corollary 3.5]{FunkeKielak2018}); this semi-norm was christened the \emph{Thurston norm} (and denoted $\| \cdot \|_T$), due to analogies with the $3$-manifold case (which we will discuss later).

For every character $\phi$ with $\im \phi = \Z$, the number $\| \phi \|_T$ is equal to minus the $L^2$-Euler characteristic of $\ker \phi$. When $\ker \phi$ is finitely generated (which happen precisely when $\{\phi, -\phi\} \subset \Sigma^1(G)$), then it is in fact a free group (as shown by Geoghegan--Mihalik--Sapir--Wise \cite[Theorem 2.6 and Remark 2.7]{Geogheganetal2001}).
Thus, for such a $\phi$, the number $\| \phi \|_T$ is equal to minus the usual Euler characteristic of $\ker \phi$ (see \cite[Theorem 1.35(2)]{Lueck2002} for a proof of this fact), that is we have
\[
 G \cong F_{1+\| \phi \|_T} \rtimes \Z
\]
with $\phi$ being equal to the projection map to $\Z$.


\begin{thm}
\label{free by cyclic}
Let $G$ be a descending HNN extension of a finitely generated non-trivial free group $F_n$. We have $\Sigma^1(G) = \Sigma^\infty(G;\Z)$, and there exists a marking of the vertices of the $L^2$-torsion
polytope $P_{L^2}(G)$
 such that $\phi \in \Sigma^1(G)$ \iff $\phi$ is marked.
\end{thm}
\begin{proof}
The group $G$ satisfies the Atiyah conjecture: it fits into a short exact sequence
\[
\langle \! \langle F_n \rangle \! \rangle \to G \to \Z
\]
with a  locally-free kernel. Hence, $G$ belongs to the class $\mathcal C_1$ of \cref{atiyah groups}. It is clear that $G$ is also torsion-free.

The group $G$ admits an obvious $2$-dimensional classifying space -- see \cite[Lemma 3.1]{FunkeKielak2018}. It is $L^2$-acyclic since it is a mapping torus, and mapping tori are $L^2$-acyclic by \cite[Theorem 1.39]{Lueck2002}.
Hence the assumptions of \cref{defic 1 atiyah} are satisfied.

We now argue as in the proof of \cref{defic 1 atiyah}. If the rank of $\fab{G}$ is at least $2$, then $P_{L^2}(G)$ is the polytope constructed in the proof of \cref{defic 1 atiyah}.

Suppose that the rank of $\fab{G}$ is $1$ (which is the smallest possible value). \cite[Corollary 3.5]{FunkeKielak2018} tells us that $P_{L^2}(G)$ is a single polytope -- it is a difference of two segments, inducing a semi-norm, and so in particular a non-negative function. Thus we may argue as in the  proof of \cref{defic 1 atiyah} even when the rank of $\fab{G}$ is $1$.
%
\end{proof}

Let us note that understanding the BNS invariants of free-by-cyclic groups is related to a large body of current research, see \cite{Dowdalletal2014,Dowdalletal2015,Dowdalletal2017} by Dowdall--I. Kapovich--Leininger and \cite{Algom-Kfiretal2015} by Algom-Kfir--Hironaka--Rafi. In particular, Dowdall--I. Kapovich--Leininger~\cite[Theorem 1.2]{Dowdalletal2017a} showed that if $G = F_n \rtimes_g \Z$ is hyperbolic, $\phi$ denotes the induced character, and  $g$ is fully irreducible, then every other integral character in the same component of $\Sigma^1(G)$ as $\phi$ also comes from a free-by-cyclic decomposition of $G$ with a fully-irreducible monodromy. In our language, this implies that if $G$ is hyperbolic then we can talk about \emph{fully-irreducible vertices} of $P_{L^2}(G)$.

\subsection{Poincar\'e duality groups}

We are moving towards (aspherical) manifolds in dimension $3$. First, however, we will discuss their homological cousins.

\begin{dfn}
A \emph{Poincar\'e duality group in dimension $n$} is a group $G$ such that there exists a $G$-module $\mathcal O$ (the \emph{orientation module}) isomorphic to $\Z$ as an abelian group, and a class $f \in H_n(G;\mathcal O)$ (the \emph{fundamental class}), such that for any $G$-module $M$ the cup product with $f$ gives an isomorphism
\[
H^i(G;M) \cong H_{n-i}(G; M \otimes \mathcal O)
\]
When $\mathcal O$ is the trivial $G$-module, we say that $G$ is \emph{orientable}.
\end{dfn}
Note that the notation above requires a $G$-bimodule structure on $M$ and $\mathcal O$. This is readily available for any right $\Z G$ module -- the left multiplication by $g$ is equal to the right multiplication by $g^{-1}$.

A natural example of a Poincar\'e duality group is the fundamental group of a closed aspherical manifold. Davis~\cite{Davis2000} constructed other examples, which however are not finitely presented. It is open whether there exist finitely presented Poincar\'e duality groups which are not fundamental groups of  closed aspherical manifolds.

Note that the fundamental class $f$ is not trivial, since $H_0(G;\Z)$ is not trivial.


The author suspects that the following is well-known, but was unable to find a reference in the literature.

\begin{prop}
\label{pd finite complex}
Let $G$ be a Poincar\'e duality group in dimension $n$ (with $n\geqslant 3$) which is of type  $\typeF{}$. There exists a free resolution of length $n$ of the trivial module $\Z$ by finitely generated $\Z G$-modules. Moreover, we can arrange for the $0^{th}$ 
term of the resolution to be of rank $1$, and the boundary map $\partial_1$ can be taken to be a row vector $(1-s_i)_i$, where $\S = \{s_1, \dots, s_k\}$ is some generating set  of $G$.
\end{prop}
\begin{proof}
Let $C = (C_\ast, \partial_\ast)$ be the cellular chain complex of the universal cover of some finite CW-classifying space $X$ of $G$. We denote the cochain complex by $(C^\ast, \partial^\ast)$ with $\partial^i \colon C^{i-1} \to C^{i}$.
By collapsing a maximal tree in the $1$-skeleton of $X$, we easily arrange for $C_0$ to be of rank $1$ as a $\Z G$-module, and for $\partial_1$ to be as desired.

Let us represent the fundamental class $f$ by a cycle $c_f$ in $C_n$. Wall in the proof of \cite[Lemma 1.1]{Wall1967} explains how to make sense of taking the cup product with a cycle; thus the cup product with $c_f$ gives us a chain map $\xi \colon C^\ast \to C_{n-\ast}$. The map $\xi$ induces the isomorphism $H_i(C^\ast \otimes M) \cong H_{n-i}(C_\ast \otimes M \otimes \mathcal O)$ for every $\Z G$-module $M$. In particular, this holds for $M = \Z G$; observe that $H_{i}(C_\ast \otimes \mathcal O) = H_{i}(C^\ast)$ for every $i$, as both coincide with the homology of the universal cover of $X$.
Since the modules $C_n$ and $C^n$ are free (and hence projective) and finitely generated, Whitehead's theorem tells us that  $\xi$ is a chain homotopy equivalence. This means that there exist chain maps $\xi' \colon C_{\ast} \to C^{n-\ast}$, $p \colon C_\ast \to C_{\ast+1}$ and $p' \colon C^\ast \to C^{\ast-1}$ such that
$
\xi \xi' = \id - \partial p - p \partial \textrm{ and } \xi' \xi = \id - \partial p' - p' \partial
$.
\[
\xymatrix{
\cdots \ar[r] & C_{i+1} \ar[r]^{\partial_{i+1}} & C_0 \ar[r]^{\partial_{i}} & C_{i-1} \ar[r] & \cdots \\
\cdots \ar[r] & C^{n-i-1} \ar[r]^{\partial^{n-i}} \ar[u]^{\xi_{n-i-1}} & C^{n-i} \ar[r]^{\partial^{n-i+1}} \ar[u]^{\xi_{n-i}} & C^{n-i+1} \ar[r] \ar[u]_{\xi_{n-i+1}} & \cdots \\
\cdots \ar[r] & C_{i+1} \ar[r]^{\partial_{i+1}} \ar[u]^{\xi'_{i+1}} & C_{i} \ar[r]^{\partial_{i}} \ar[u]_{\xi'_{i}} \ar@{.>}[uul]|<<<<<<<<{p_i} & C_{i-1} \ar[r] \ar[u]_{\xi'_{i-1}} \ar@{.>}[uul]|<<<<<<<<{p_{i-1}} & \cdots \\
\cdots \ar[r] & C^{n-i-1} \ar[r]^{\partial^{n-i}} \ar[u]^{\xi_{n-i-1}} & C^{n-i} \ar[r]^{\partial^{n-i+1}} \ar[u]_{\xi_{n-i}} \ar@{.>}[uul]|<<<<<<<<{p'_{n-i}} & C^{n-i+1} \ar[r] \ar[u]_{\xi'_{i-1}} \ar@{.>}[uul]|<<<<<<<<{p'_{n-i+1}} & \cdots
}
\]

Let $m$ be maximal such that $C_m \neq 0$.
Note that $m\geqslant n$, since $H_n(G; \Z) \neq 0$, as it contains $f$.
We claim that we can modify the chain complex $C$ so that $m=n$, and the new complex is still a free resolution of $\Z$ by finitely generated $\Z G$-modules.

If $m=n$ then we are done. Suppose that $m>n$.
We modify $C$ to form a new chain complex, $D$, as follows: $D$ is identical to $C$ except $D_m = 0$, and $D_{m-2} = C_{m-2} \oplus C_m$; the boundary map $D_{m-1} \to D_{m-2}$ is equal to the transpose $(\partial_{m-1}, p_{m-1})^T$, and the boundary map $D_{m-2} \to D_{m-3}$ is equal to $(\partial_{m-2}, 0)$.
\[
\xymatrixcolsep{3pc}
\xymatrix{
 C = & C_{m} \ar[r]^{\partial_{m}} & C_{m-1} \ar[r]^{\partial_{m-1}} & C_{m-2} \ar[r]^{\partial{m-2}} & C_{m-3} \ar[r] & \cdots \\
 D = & 0 \ar[r] & C_{m-1} \ar[r]^-{(\partial_{m-1}, p_{m-1})^T } \ar@{.>}[lu]^{p_{m-1}} & C_{m-2}\oplus C_m \ar[r]^-{(\partial_{m-2},0)} &  C_{m-3} \ar[r] & \cdots
}
\]
It is immediate that $D$ is a chain complex of finitely generated free modules; it is also immediate that it is exact everywhere, except possibly at $D_{m-1}$ and $D_{m-2}$.

Crucially, since $C^{n-m} = 0$, we have $\xi_{n-m} = 0$ and so $p_{m-1} \partial_m = \id$.
Let $x$ be a cycle in $D_{m-1} = C_{m-1}$. By exactness of $C$, we see that $x \in \im \partial_{m}$; let $y$ denote a preimage of $x$ in $C_m$. Since $x$ is a cycle, we have $p_{m-1}(x) = 0$. But $p_{m-1}(x) = p_{m-1} \partial_m(y) = y$, and so $y=0$ and thus $x=0$.

Now let $x \in D_{m-2}$ be a cycle. We write $x = (x_0,x_1)^T \in C_{m-2} \oplus C_{m}$; being a cycle means that $x_0$ is a cycle in $C_{m-2}$, and so there exists $y \in C_{m-1}$ such that $x_0 = \partial_{m-1}(y)$.
We can pick $y$ so that $p_{m-1}(y) = 0$, using exactness of $C$ and the fact that $p_{m-1} \partial_m = \id$.
Now
\[
(\partial_{m-1},p_{m-1})^T(y+\partial_m(x_1)) = (x_0, p_{m-1}(y) + p_{m-1} \partial_m(x_1) )^T = (x_0,x_1)^T = x
\]
since $p_{m-1}(y) = 0$ (by construction) and $p_{m-1} \partial_m = \id$.

We replace $C$ by $D$;
note that we can still define the chain homotopy equivalence $\xi$ between $D$ and its dual cochain complex. Alternatively, we can realise $D$ as the cellular chain complex of the universal cover of a finite CW-classifying space of $G$: we can remove each $m$-cell from $X$ and attach an $(m-2)$ sphere instead of it. Then we modify the attaching maps of the $(m-1)$ cells so that they coincide (in the universal cover) with the boundary map in $D$. The resulting space is aspherical, since $D$ is exact away from degree $0$. Since $m > n \geqslant 3$, the fundamental group did not change under these modifications.

Note that when passing from $C$ to $D$ we did not alter $C_1$ nor the boundary map $\partial_1 \colon C_1 \to C_0$.
We now repeat the argument until the length of the resolution is $n$.
\end{proof}

\begin{thm}
\label{pd}
Let $G$ be a Poincar\'e duality group in dimension $3$. Suppose further that $G$ is agrarian and of type $\typeF{}$.
We have $\Sigma^\infty(G;\Z)=  \Sigma^1(G)$, and there exists a marked polytope $P$ such that for every $\phi \in H^1(G;\R) \s-\{0\}$ we have   $\phi \in \Sigma^1(G)$ \iff $\phi$ is marked.
\end{thm}
\begin{proof}
\cref{pd finite complex} tells us that there exists a free resolution $C = (C_\ast, \partial_\ast)$ of length $3$ of the trivial $\Z G$-module $\Z$ by finitely generated $\Z G$-modules. We also know that $C_0$ has rank $1$, and $\partial_1 = (1-s)_{i}$, where $\S = \{s_1, \dots, s_k \}$ is a finite generating set of $G$.

Take $\phi \in H^1(G;\R) \s- \{0\}$, and take $s \in \S$ with $\phi(s) \neq 0$. Let $U_\phi$ denote a closed neighbourhood of $\phi$ such that for every $\psi \in U_\phi$ we have $\psi(s) \neq 0$. Assume further that $U_\phi$ is the convex hull of finitely many characters. It is immediate that $\partial_1 \otimes \widehat{\Z G}^{U_\phi}$ is onto (since $1-s$ is invertible in $\widehat{\Z G}^{U_\phi}$), and so $H_0(C \otimes \widehat{\Z G}^{U_\phi}) = 0$.
Poincar\'e duality tells us that the transpose of $\partial_3 \otimes \widehat{\Z G}^{U_\phi}$ is onto as well. Thus, since $C_3 \otimes \widehat{\Z G}^{U_\phi}$ is free, we see that it is a summand of $C_2 \otimes \widehat{\Z G}^{U_\phi}$. Therefore we have
\[
C_2 \otimes \widehat{\Z G}^{U_\phi} = C_3 \otimes \widehat{\Z G}^{U_\phi} \oplus C_2'
\]
for some (stably free) $\widehat {\Z G}^{U_\phi}$-module $C_2'$.

We will modify the complex $C$ by taking the direct sum with the complex
\[
\xymatrix{
 C_3 \ar[r]^\id & C_3 \ar[r] & 0
}
\]
(where the $0$ module lies in degree $0$).
It is clear that the new chain complex $D = (D_\ast, \delta_\ast)$ is a free resolution of $\Z$ by finitely generated $\Z G$-modules, and thus it computes the homology of $g$ in the same way as $C$ did. 

Let $k_i$ denote the rank of the free module $D_i$; we have already mentioned that $k_0=1$. Consider the homology of $D \otimes \D$. Since $\D$ is a skew-field, we see that the Euler characteristic of $D \otimes \D$ is $k_0 - k_1 + k_2 - k_3$. But, since $G$ is a Poincar\'e duality group of odd dimension, it is immediate that this Euler characteristic must be equal to $0$. So
\[
k_1 + k_3 = k_2 + 1
\]

We have
\begin{align*}
D_2 \otimes \widehat{\Z G}^{U_\phi} &= C_2 \otimes \widehat{\Z G}^{U_\phi} \oplus C_3 \otimes \widehat{\Z G}^{U_\phi} \\
&= \big( C_3 \otimes \widehat{\Z G}^{U_\phi} \oplus C_2' \big) \oplus C_3 \otimes \widehat{\Z G}^{U_\phi} \\
&\cong C_3 \otimes \widehat{\Z G}^{U_\phi} \oplus C_2 \otimes \widehat{\Z G}^{U_\phi}
\end{align*}
The last isomorphism implies the existence of an invertible matrix $M$ over $\widehat{\Z G}^{U_\phi}$ (the change of basis matrix of $D_2 \otimes \widehat{\Z G}^{U_\phi}$), such that
$M \delta_3$ is the identity $k_3 \times k_3$ matrix extended at the bottom by the $(k_2-k_3) \times k_3$ zero matrix.

Now we will replace the matrix $M$ by a matrix $M_0$ over $\Z G$ by truncating the entries of $M$ (the argument is very similar  to the one used in \cref{entries far away}). Recall that $U_\phi$ is the convex hull of finitely many characters, say $\rho_1, \dots, \rho_l$. Pick $C \in \R$ such that $C + \rho_i(g) > 0$ for every $i$ and every $g \in G$ appearing in the support of some entry of $\delta_3$ (which is a matrix over $\Z G$). Now, let $M_0$ denote a matrix over $\Z G$ obtained from $M$ by truncating the entries of $M$ to the subset $\bigcup_i \rho_i^{-1}\big( (0,C) \big)$ of $G$ -- formally speaking, we treat the entries of $M$ as functions $G \to \Z$, we first restrict these functions to $\bigcup_i \rho_i^{-1}\big( (0,C) \big)$ and then extend back to functions from $G$ by setting their values outside of $\bigcup_i \rho_i^{-1}\big( (0,C) \big)$ to be $0$. By the choice of $C$, we see that
$M_0 \delta_3$ is a $\rho_i$-identity for every $i$. Now, a simple convexity argument immediately shows that in fact
$M_0 \delta_3$ is a $\psi$-identity for every $\psi \in U_\phi$.

The fact that $M_0 \delta_3$ is a $\psi$-identity implies that the span of the last $(k_2-k_3)$ basis vectors of $D_2\otimes \widehat{\Z G}^{U_\phi}$ is a complement to the image of the (injective) map $M_0 \delta_3 \colon D_3\otimes\widehat{\Z G}^{U_\phi} \to D_2\otimes \widehat{\Z G}^{U_\phi}$. Therefore, when computing the homology of $D \otimes  \widehat{\Z G}^{U_\phi}$, we can disregard $D_3$ and replace $D_2$ by this complement $D'_2$.
This already tells us that $H_i(G;\widehat {\Z G}^\psi) = 0$ for every $i>2$ and every $\psi \in U_\phi$.

Crucially, the differential $D'_2 \to D_1$ is a matrix defined over $\Z G$, say $A$. The matrix $A$ is a $k_1 \times (k_2-k_3)$ matrix. But $k_2 - k_3 = k_1 - 1$, and so $A$ has one row more than column. We now form $A_\phi$ by removing the row of $A$ corresponding to $s$, that is, if $s=s_i$ then we remove the $i^{th}$ row. Again, $A_\phi$ is defined over $\Z G$, and
$\phi \in \Sigma^1(G)$ \iff  $A_\phi$ is right-invertible. But, by definition, $A_\phi$ is right-invertible \iff $\phi \in \Sigma(A_\phi)$. Moreover, $\phi \in \Sigma^2(G;\Z)$ \iff $\phi \in \Sigma^1(G)$ and $A_\phi$ is injective. We have already seen that a result of Bieri~\cite[Theorem 3]{Bieri2007} tells us that if $A_\phi$ is right-invertible then it is also injective. Hence we obtain: $\phi \in \Sigma^1(G)$ \iff $\phi \in \Sigma^2(G;\Z)$ \iff $\phi \in \Sigma(A_\phi)$. We have also already shown that higher homology groups of $G$ with $\widehat{\Z G}^{\phi}$ coefficients are trivial, and so $\Sigma^1(G) = \Sigma^\infty(G)$.

By \cref{bns matrix}, $\phi \in \Sigma(A_\phi)$ is equivalent to $\phi$ being marked by a marking of $P(A_\phi)$. Now, we observe that there are finitely many characters $\phi_1, \dots, \phi_n$ such that every character in $H^1(G;\R) \s- \{0\}$ is contained in at least one of the neighbourhoods $U_{\phi_i}$ up to positive homothety -- this follows immediately from the compactness of the unit sphere in $H^1(G;\R)$. Therefore, $U_{\phi_1}, \dots, U_{\phi_n}$ and $P(A_{\phi_1}), \dots, P(A_{\phi_n})$ form an atlas of markings, and we are done after an application of \cref{one marked polytope}, which combines the marked polytopes $P(A_{\phi_i})$ into a single marked polytope $P$.
\end{proof}

Note that if we were to strengthen the hypothesis by assuming $G$ to satisfy the Atiyah conjecture, it is not clear whether we would be able to extract a stronger result -- it is not even clear that the $L^2$-torsion polytope (even though it does exist) would be a single polytope; more importantly for us, it is not clear what the contribution to the polytope coming from the top-degree differential is. In the case of an honest $3$-manifold, it is the same as the contribution of the differential of the lowest degree, but this does not seem to follow from Poincar\'e duality alone.

\subsection{\texorpdfstring{$3$}{3}-manifolds}
\label{sec 3mfld}

We will now give a new proof of the following theorem of Thurston.

 \begin{thm}[Thurston {\cite[Theorem 5]{Thurston1986}}]
 \label{thurstons thm}
  Let $M$ be a compact, oriented $3$-manifold. The set $F$ of cohomology classes in $H^1(M)$ representable by non-singular closed $1$-forms is some union of the cones on open faces of $B_x$, minus the origin. The set of elements in $H^1(M;\Z)$ whose Lefschetz dual is represented by a fibre of a fibration consists of the set of all lattice points in $F$.
 \end{thm}

The object $B_x$ appearing above is the unit ball of the \emph{Thurston semi-norm} $x \colon H^1(M;\R) \to [0,\infty)$.
The semi-norm is first defined on the integral characters, then defined on positive multiples of such characters by the formula $x(\lambda \phi) = \lambda x(\phi)$, and lastly it is extended by continuity to the whole of $H^1(M;\R)$.
Given an integral character $\phi$ (that is a character with $\im \phi = \Z$, the standard copy of $\Z$ in $\R$), we define
\[
 x(\phi) = \min_\Sigma -\chi_-(\Sigma)
\]
where $\Sigma$ runs through all embedded (not necessarily connected) oriented surfaces in $M$ dual to $\phi$ via Poincar\'e--Lefschetz duality, and $\chi_-(\Sigma)$ denotes the sum of the Euler characteristics of the connected components of $\Sigma$ which are not spheres.

Thurston also defined the polytope $B_{x^\ast} \subseteq H_1(M;\R)$ dual to $B_x$; he proved in \cite[Theorem 2]{Thurston1986} that $B_{x^\ast}$ is an integral polytope. He also proved in \cite[Theorem 3]{Thurston1986} that the open faces of $B_{x}$ forming the union of cones $F$ are all faces of maximal dimension (called \emph{fibred faces}) -- on the level of the dual polytope $B_{x^\ast}$ this translates to the statement that $F$ is the union of some duals of vertices of $B_{x^\ast}$.

The identification of the set of elements in $H^1(M;\Z)$ whose Lefschetz dual is represented by a fibre of a fibration with the set of all (integral) lattice points in $F$ follows immediately from the work of Tischler~\cite{Tischler1970}: one has to observe that (topological) $3$-manifolds can always be taken to be smooth -- it was shown by Moise~\cite[Theorems 1 and 3]{Moise1952} that $3$-manifolds can be triangulated in such a way that closed stars of vertices are piecewise-linearly homeomorphic to a simplex in $\R^3$. This is enough to construct a smooth structure.

Now let us argue that without loss of generality one may take $M$ to be connected in the statement of Thurston's theorem: suppose that we have shown the result for connected manifolds, and let $M = M_1 \sqcup \dots \sqcup M_m$ be a decomposition into connected components. Note that $H^1(M;\R) = \prod_i H^1(M_i;\R)$, and the Thurston norm $x$ restricted to $H^1(M_i;\R)$ is equal to the Thurston norm $x_i$ on $M_i$. Thus, $B_{x^\ast}$ is simply the product of the polytopes $B_{x_i^\ast}$, and duals of faces of $B_{x^\ast}$ are products of duals of corresponding faces of $B_{x_i^\ast}$.
Now, a cohomology class in $H^1(M;\R)$ is representable by a non-singular closed $1$-form \iff its component in every $H^1(M_i;\R)$ is.

When $M$ is connected,
Bieri--Neumann--Strebel observed in \cite[Corollary F]{Bierietal1987} that $F  = \Sigma^1(\pi_1(M))$, thus connecting Thurston's result to the BNS invariants.

Let us summarise this discussion in a statement that is equivalent to Thurston's theorem above.

 \begin{thm}[Reformulation of Thurston's theorem]
  Let $M$ be a compact, connected, oriented $3$-manifold, and let $G = \pi_1(M)$. There exists a marking of vertices of $B_{x^\ast}$ such that for every $\phi \in H^1(G;\R) \s- \{0\}$ we have $\phi \in \Sigma^1(G)$ \iff $\phi$ is marked.
 \end{thm}
 \begin{proof}
If $\Sigma^1(G) = \emptyset$, then the statement holds vacuously. Let us assume that $\Sigma^1(G) \neq \emptyset$.

We will first argue that $G$ satisfies the Atiyah conjecture. Since $\Sigma^1(G)$ is open (by \cref{bns open orig}), it contains an integral surjective character $\phi \colon G \to \Z$, and so $M$ fibres smoothly over the circle (by \cite{Tischler1970}). The fibre must be a compact orientable surface. Irrespectively of whether it is closed or not, its fundamental group is residually \{torsion-free nilpotent\} (see e.g. \cite[Corollary 2]{Baumslag2010}). Note that residual nilpotence is equivalent to saying that the intersection of all the terms in the lower central series is trivial. But the lower central series consists of characteristic subgroups, and so any extension of a residually nilpotent group by $\Z$ is residually solvable. Moreover, every extension of a residually \{torsion-free nilpotent\} group by $\Z$ is residually \{torsion-free solvable\}. Such groups satisfy the Atiyah conjecture -- they are torsion free elements of the class $\mathcal C_2$ of \cref{atiyah groups}. In particular, our fundamental group $G$ satisfies the Atiyah conjecture.

\smallskip
By a result of L\"uck~\cite[Theorem 1.39]{Lueck2002}, $M$ is $L^2$-acyclic, as it is a mapping torus over a finite CW complex.
Also, since $M$ fibres with fibre an oriented surface, it is immediate that either $M$ is homeomorphic to $\mathbb S^2 \times \mathbb S^1$ (as there are no non-trivial fibrations over $\mathbb S^1$ with fibre $\mathbb S^2$), or $M$ is aspherical, since the circle and any oriented surface which is not $\mathbb S^2$ are aspherical. In the latter case it is also clear that the boundary of $M$ is either empty or a collection of tori.

When $M \cong \mathbb S^1 \times \mathbb S^2$, then $\Sigma^1(G) = H^1(G;\R) \s- \{0\}$, and we declare all faces of $B_{x^\ast}$ marked. One can compute $B_{x^\ast}$ directly, and it turns out to be a single point, the origin of $H^1(G;\R)$. Thus $B_{x^\ast}$ is a polytope with marked vertices.

\smallskip
We are now ready for the main case: $M$ is an $L^2$-acyclic aspherical manifold with empty or toroidal boundary whose fundamental group satisfies the Atiyah conjecture. At this point we can already use \cref{pd}, but since we have the Atiyah conjecture at our disposal and the manifold structure, we will actually learn more about the marked polytope occurring in \cref{pd}.

Pick a finite CW-structure for $M$, and let $C$ denote the cellular chain complex of the universal cover of $M$. The complex $C$ is a free $\Z G$-resolution of the trivial module $\Z$, since $M$ is aspherical.

Suppose first that $M$ is closed.
McMullen in the proof of~\cite[Theorem 5.1]{McMullen2002} shows that we may take $C$ as follows:
\[
\xymatrix{
C= & C_3 \ar[r]^{\partial_3} & C_2 \ar[r]^{\partial_2} & C_1 \ar[r]^{\partial_1} & C_0
}
\]
where $C_0$ and $C_3$ are of rank one, $C_1$ and $C_2$ are of the same rank, $\partial_3$ is the transpose of $\partial_1$, which in turn is the vector $(1-s)_{s \in \S}$ with $\S$ being a finite generating set of $G$ as usual. For any $\phi \in H^1(G;\R) \s- \{0\}$ there exists $s_i \in \S$ with $\phi(s_i) \neq 0$.
Now let $A_i$ denote the matrix obtained from $\partial_2$ by removing the $i^{th}$ row and the $i^{th}$ column. The matrix $A_i$ is a square matrix over $\Z G$, and the structure of $\partial_1$ and $\partial_3$, together with the invertibility of $1-s_i$ over $\Z G^\phi$, immediately imply that $H_1(G;\widehat {\Z G}^\phi)=0$ \iff $\phi \in \Sigma(A_i)$.
Thus $\Sigma(A_i)$ induces a marking of the vertices of $P(A_i)$, using \cref{bns matrix} as before.

Since $G$ is $L^2$-acyclic, the $L^2$-torsion polytope $P_{L^2}(G)$ is defined, and we have $P_{L^2}(G) = P(A_i) - 2 P(1-s_i)$ -- this follows from  additivity of the universal $L^2$-torsion, see \cite[Lemma 1.9]{FriedlLueck2017}. In fact, $P_{L^2}(G)$ is a single polytope:
 if the rank of $\fab{G}$ is at least $2$, it follows from \cref{flat polys}; otherwise it follows from \cite[Theorem 2.39]{FriedlLueck2017}, since the induced function $\| \cdot \|_{P_{L^2}(G)}$ can be a semi-norm only when $P_{L^2}(G)$ is a single polytope (in the case of $\fab{G}$ of rank $1$).
Now we need to define a marking on the vertices of $P_{L^2}(G)$, using the given markings on the polytopes $P(A_i)$.
This is achieved by \cref{one marked L2 polytope}, taking $U_i = \{\psi \in H^1(G;\R) \s- \{0\} \mid \psi(s_i) \neq 0 \}$.

\smallskip
Now suppose that $M$ has boundary.
Again, using the proof of~\cite[Theorem 5.1]{McMullen2002}, we see that
\[
\xymatrix{
C= & C_2 \ar[r]^{\partial_2} & C_1 \ar[r]^{\partial_1} & C_0
}
\]
where $C_1$ is of rank one greater than $C_2$. We argue precisely as before, reducing the problem of whether $\phi \in \Sigma^1(G)$ to whether the matrix $A_i$ obtained by removing the column of $\partial_2$ corresponding to $s_i$ admits a right-inverse over $\widehat {\Z G}^\phi$ (as usual, we have $\phi(s_i) \neq 0$).

In this case we have $P_{L^2}(G) = P(A_i) - P(1-s_i)$, and the argument proceeds exactly as before.

It was shown by Friedl--L\"uck~\cite[Theorem 3.37]{FriedlLueck2017} that $P_{L^2}(G)$ coincides with $B_{x^\ast}$. This concludes the proof.
\end{proof}

It is immediate from Thurston's construction that the equality of semi-norms $x = \| \cdot \|_{B_{x^\ast}}$ is satisfied. Friedl--L\"uck proved that $B_{x^\ast} = P_{L^2}(G)$ when $M$ is aspherical, not homeomorphic to $\mathbb S^1 \times \mathbb S^2$, and $G$ satisfies the Atiyah conjecture. This equality yields a new point of view on the value $x(\phi)$ for $\phi$ with $\im \phi = \Z$: when $\ker \phi$ is finitely generated, it was known that $x(\phi)$ gives minus the Euler characteristic of the surface, which is the fibre of the corresponding fibration. When $\ker \phi$ is not finitely generated, we see now that $x(\phi)$ gives minus the $L^2$-Euler characteristic of the kernel.

\bibliographystyle{math}
\bibliography{bibliography}

\begin{thebibliography}{GMSW}

\bibitem[Ago]{Agol2013}
Ian Agol.
\newblock {The virtual {H}aken conjecture}.
\newblock {\em Doc. Math.} {\bf 18}(2013), 1045--1087.
\newblock With an appendix by Agol, Daniel Groves, and Jason Manning.

\bibitem[AKHR]{Algom-Kfiretal2015}
Yael Algom-Kfir, Eriko Hironaka, and Kasra Rafi.
\newblock {Digraphs and cycle polynomials for free-by-cyclic groups}.
\newblock {\em Geom. Topol.} {\bf 19}(2015), 1111--1154.

\bibitem[Alm]{Almeida2017}
Kisnney Almeida.
\newblock {The {BNS}-invariant for some {A}rtin groups of arbitrary circuit
  rank}.
\newblock {\em J. Group Theory} {\bf 20}(2017), 793--806.

\bibitem[AK1]{AlmeidaKochloukova2015a}
Kisnney Almeida and Dessislava Kochloukova.
\newblock {The {$\Sigma^1$}-invariant for {A}rtin groups of circuit rank 1}.
\newblock {\em Forum Math.} {\bf 27}(2015), 2901--2925.

\bibitem[AK2]{AlmeidaKochloukova2015}
Kisnney Almeida and Dessislava Kochloukova.
\newblock {The {$\Sigma^1$}-invariant for some {A}rtin groups of rank 3
  presentation}.
\newblock {\em Comm. Algebra} {\bf 43}(2015), 702--718.

\bibitem[Bar]{Bartholdi2016}
Laurent Bartholdi.
\newblock {Amenability of groups is characterized by {M}yhill's {T}heorem}.
\newblock {\em J. Eur. Math. Soc.} (2019).
\newblock With an appendix by Dawid Kielak.

\bibitem[Bau]{Baumslag2010}
Gilbert Baumslag.
\newblock {Some reflections on proving groups residually torsion-free
  nilpotent. {I}}.
\newblock {\em Illinois J. Math.} {\bf 54}(2010), 315--325.

\bibitem[Bie]{Bieri2007}
Robert Bieri.
\newblock {Deficiency and the geometric invariants of a group}.
\newblock {\em J. Pure Appl. Algebra} {\bf 208}(2007), 951--959.
\newblock With an appendix by Pascal Schweitzer.

\bibitem[BG]{BieriGroves1984}
Robert Bieri and J.~R.~J. Groves.
\newblock {The geometry of the set of characters induced by valuations}.
\newblock {\em J. Reine Angew. Math.} {\bf 347}(1984), 168--195.

\bibitem[BNS]{Bierietal1987}
Robert Bieri, Walter~D. Neumann, and Ralph Strebel.
\newblock {A geometric invariant of discrete groups}.
\newblock {\em Invent. Math.} {\bf 90}(1987), 451--477.

\bibitem[BR]{BieriRenz1988}
Robert Bieri and Burkhardt Renz.
\newblock {Valuations on free resolutions and higher geometric invariants of
  groups}.
\newblock {\em Comment. Math. Helv.} {\bf 63}(1988), 464--497.

\bibitem[Bro]{Brown1987}
Kenneth~S. Brown.
\newblock {Trees, valuations, and the {B}ieri-{N}eumann-{S}trebel invariant}.
\newblock {\em Invent. Math.} {\bf 90}(1987), 479--504.

\bibitem[CL]{CashenLevitt2016}
Christopher~H. Cashen and Gilbert Levitt.
\newblock {Mapping tori of free group automorphisms, and the
  {B}ieri--{N}eumann--{S}trebel invariant of graphs of groups}.
\newblock {\em J. Group Theory} {\bf 19}(2016), 191--216.

\bibitem[CDKV]{Cavalloetal2017}
Bren Cavallo, Jordi Delgado, Delaram Kahrobaei, and Enric Ventura.
\newblock {Algorithmic recognition of infinite cyclic extensions}.
\newblock {\em J. Pure Appl. Algebra} {\bf 221}(2017), 2157--2179.

\bibitem[Coh]{Cohn1977}
Paul~Moritz Cohn.
\newblock {\em Skew field constructions}.
\newblock Cambridge University Press, Cambridge-New York-Melbourne, 1977.
\newblock London Mathematical Society Lecture Note Series, No. 27.

\bibitem[Dav]{Davis2000}
Michael~W. Davis.
\newblock {Poincar\'e duality groups}.
\newblock In {\em Surveys on surgery theory, {V}ol. 1}, volume 145 of {\em Ann.
  of Math. Stud.}, pages 167--193. Princeton Univ. Press, Princeton, NJ, 2000.

\bibitem[Del]{Delzant2010}
Thomas Delzant.
\newblock {L'invariant de {B}ieri-{N}eumann-{S}trebel des groupes fondamentaux
  des vari\'et\'es k\"ahl\'eriennes}.
\newblock {\em Math. Ann.} {\bf 348}(2010), 119--125.

\bibitem[DNR]{Deroinetal2016}
B.~Deroin, A.~Navas, and C.~Rivas.
\newblock {\em Groups, Orders, and Dynamics}.
\newblock ar{X}iv:1408.5805v2.

\bibitem[Die]{Dieudonne1943}
Jean Dieudonn{\'e}.
\newblock {Les d\'eterminants sur un corps non commutatif}.
\newblock {\em Bull. Soc. Math. France} {\bf 71}(1943), 27--45.

\bibitem[DKL1]{Dowdalletal2014}
Spencer Dowdall, Ilya Kapovich, and Christopher~J. Leininger.
\newblock {Unbounded asymmetry of stretch factors}.
\newblock {\em C. R. Math. Acad. Sci. Paris} {\bf 352}(2014), 885--887.

\bibitem[DKL2]{Dowdalletal2015}
Spencer Dowdall, Ilya Kapovich, and Christopher~J. Leininger.
\newblock {Dynamics on free-by-cyclic groups}.
\newblock {\em Geom. Topol.} {\bf 19}(2015), 2801--2899.

\bibitem[DKL3]{Dowdalletal2017a}
Spencer Dowdall, Ilya Kapovich, and Christopher~J. Leininger.
\newblock {Endomorphisms, train track maps, and fully irreducible monodromies}.
\newblock {\em Groups Geom. Dyn.} {\bf 11}(2017), 1179--1200.

\bibitem[DKL4]{Dowdalletal2017}
Spencer Dowdall, Ilya Kapovich, and Christopher~J. Leininger.
\newblock {Mc{M}ullen polynomials and {L}ipschitz flows for free-by-cyclic
  groups}.
\newblock {\em J. Eur. Math. Soc. (JEMS)} {\bf 19}(2017), 3253--3353.

\bibitem[DT]{DowdallTaylor2018}
Spencer Dowdall and Samuel~J. Taylor.
\newblock {Hyperbolic extensions of free groups}.
\newblock {\em Geom. Topol.} {\bf 22}(2018), 517--570.

\bibitem[Dro]{Droms1983}
Carl Gordon~Arthur Droms.
\newblock {\em Graph groups}.
\newblock ProQuest LLC, Ann Arbor, MI, 1983.
\newblock Thesis (Ph.D.)--Syracuse University.

\bibitem[FL]{FriedlLueck2017}
Stefan Friedl and Wolfgang L{\"u}ck.
\newblock {Universal {$L^2$}-torsion, polytopes and applications to
  3-manifolds}.
\newblock {\em Proc. Lond. Math. Soc. (3)} {\bf 114}(2017), 1114--1151.

\bibitem[FLT]{Friedletal2016}
Stefan Friedl, Wolfgang L\"uck, and Stephan Tillmann.
\newblock {Groups and Polytopes}.
\newblock {\em ar{X}iv:1611.01857}.

\bibitem[FT]{FriedlTillmann2015}
Stefan Friedl and Stephan Tillmann.
\newblock {Two-generator one-relator groups and marked polytopes}.
\newblock {\em {arXiv}:1501.03489}.

\bibitem[Fun]{Funke2018}
Florian Funke.
\newblock {The {$L^2$}-torsion polytope of amenable groups}.
\newblock {\em Doc. Math.} {\bf 23}(2018), 1969--1993.

\bibitem[FK]{FunkeKielak2018}
Florian Funke and Dawid Kielak.
\newblock {Alexander and {T}hurston norms, and the
  {B}ieri--{N}eumann--{S}trebel invariants for free-by-cyclic groups}.
\newblock {\em Geom. Topol.} {\bf 22}(2018), 2647--2696.

\bibitem[GM]{GarilleMeier1998}
Susan~Garner Garille and John Meier.
\newblock {Whitehead graphs and the {$\Sigma^1$}-invariants of infinite
  groups}.
\newblock {\em Internat. J. Algebra Comput.} {\bf 8}(1998), 23--34.

\bibitem[GMSW]{Geogheganetal2001}
Ross Geoghegan, Michael~L. Mihalik, Mark Sapir, and Daniel~T. Wise.
\newblock {Ascending {HNN} extensions of finitely generated free groups are
  {H}opfian}.
\newblock {\em Bull. London Math. Soc.} {\bf 33}(2001), 292--298.

\bibitem[Hil]{Hillman1991}
Jonathan~A. Hillman.
\newblock {On {$4$}-manifolds homotopy equivalent to surface bundles over
  surfaces}.
\newblock {\em Topology Appl.} {\bf 40}(1991), 275--286.

\bibitem[KMM]{Kobanetal2015}
Nic Koban, Jon McCammond, and John Meier.
\newblock {The {BNS}-invariant for the pure braid groups}.
\newblock {\em Groups Geom. Dyn.} {\bf 9}(2015), 665--682.

\bibitem[KP]{KobanPiggott2014}
Nic Koban and Adam Piggott.
\newblock {The {B}ieri-{N}eumann-{S}trebel invariant of the pure symmetric
  automorphisms of a right-angled {A}rtin group}.
\newblock {\em Illinois J. Math.} {\bf 58}(2014), 27--41.

\bibitem[Koc1]{Kochloukova2006}
D.~H. Kochloukova.
\newblock {Some Novikov rings that are von Neumann finite}.
\newblock {\em Comment. Math. Helv.} {\bf 81}(2006), 931--943.

\bibitem[Koc2]{Kochloukova2010}
Dessislava~H. Kochloukova.
\newblock {On subdirect products of type {${\rm FP}_m$} of limit groups}.
\newblock {\em J. Group Theory} {\bf 13}(2010), 1--19.

\bibitem[LOS]{Linnelletal2012}
Peter Linnell, Boris Okun, and Thomas Schick.
\newblock {The strong {A}tiyah conjecture for right-angled {A}rtin and
  {C}oxeter groups}.
\newblock {\em Geom. Dedicata} {\bf 158}(2012), 261--266.

\bibitem[Lin]{Linnell1993}
Peter~A. Linnell.
\newblock {Division rings and group von {N}eumann algebras}.
\newblock {\em Forum Math.} {\bf 5}(1993), 561--576.

\bibitem[L{\"u}c]{Lueck2002}
Wolfgang L{\"u}ck.
\newblock {\em {$L^2$}-invariants: theory and applications to geometry and
  {$K$}-theory}, volume~44 of {\em Ergebnisse der Mathematik und ihrer
  Grenzgebiete. 3. Folge. A Series of Modern Surveys in Mathematics [Results in
  Mathematics and Related Areas. 3rd Series. A Series of Modern Surveys in
  Mathematics]}.
\newblock Springer-Verlag, Berlin, 2002.

\bibitem[Mal]{Malcev1948}
A.~I. Mal'cev.
\newblock {On the embedding of group algebras in division algebras}.
\newblock {\em Doklady Akad. Nauk SSSR (N.S.)} {\bf 60}(1948), 1499--1501.

\bibitem[McM]{McMullen2002}
Curtis~T. McMullen.
\newblock {The {A}lexander polynomial of a 3-manifold and the {T}hurston norm
  on cohomology}.
\newblock {\em Ann. Sci. \'Ecole Norm. Sup. (4)} {\bf 35}(2002), 153--171.

\bibitem[MV]{MeierVanWyk1995}
John Meier and Leonard VanWyk.
\newblock {The {B}ieri-{N}eumann-{S}trebel invariants for graph groups}.
\newblock {\em Proc. Lond. Math. Soc. (3)} {\bf 71}(1995), 263--280.

\bibitem[Moi]{Moise1952}
Edwin~E. Moise.
\newblock {Affine structures in {$3$}-manifolds. {V}. {T}he triangulation
  theorem and {H}auptvermutung}.
\newblock {\em Ann. of Math. (2)} {\bf 56}(1952), 96--114.

\bibitem[Neu]{Neumann1949}
B.~H. Neumann.
\newblock {On ordered division rings}.
\newblock {\em Trans. Amer. Math. Soc.} {\bf 66}(1949), 202--252.

\bibitem[OW]{OlliverWise2011}
Yann Ollivier and Daniel~T. Wise.
\newblock {Cubulating random groups at density less than {$1/6$}}.
\newblock {\em Trans. Amer. Math. Soc.} {\bf 363}(2011), 4701--4733.

\bibitem[OK]{Orlandi-Korner2000}
Lisa~A. Orlandi-Korner.
\newblock {The {B}ieri-{N}eumann-{S}trebel invariant for basis-conjugating
  automorphisms of free groups}.
\newblock {\em Proc. Amer. Math. Soc.} {\bf 128}(2000), 1257--1262.

\bibitem[Pas]{Passman1985}
Donald~S. Passman.
\newblock {\em The algebraic structure of group rings}.
\newblock Robert E. Krieger Publishing Co., Inc., Melbourne, FL, 1985.
\newblock Reprint of the 1977 original.

\bibitem[Sch1]{Schick2002}
Thomas Schick.
\newblock {Erratum: ``{I}ntegrality of {$L^2$}-{B}etti numbers''}.
\newblock {\em Math. Ann.} {\bf 322}(2002), 421--422.

\bibitem[Sch2]{Schreve2014}
Kevin Schreve.
\newblock {The strong {A}tiyah conjecture for virtually cocompact special
  groups}.
\newblock {\em Math. Ann.} {\bf 359}(2014), 629--636.

\bibitem[Sik]{Sikorav1987}
J.-C. Sikorav.
\newblock {Homologie de {N}ovikov associ\'ee \`a une classe de cohomologie
  r\'eelle de degr\'e un}.
\newblock {\em Th\`ese Orsay} (1987).

\bibitem[Ste]{Stein1992}
Melanie Stein.
\newblock {Groups of piecewise linear homeomorphisms}.
\newblock {\em Trans. Amer. Math. Soc.} {\bf 332}(1992), 477--514.

\bibitem[Tam]{Tamari1957}
Dov Tamari.
\newblock {A refined classification of semi-groups leading to generalized
  polynomial rings with a generalized degree concept}.
\newblock In {\em Proceedings of the {I}nternational {C}ongress of
  {M}athematicians, 1954, Vol. 1, {A}msterdam}, pages 439--440. North-Holland,
  1957.

\bibitem[Thu]{Thurston1986}
William~P. Thurston.
\newblock {A norm for the homology of {$3$}-manifolds}.
\newblock {\em Mem. Amer. Math. Soc.} {\bf 59}(1986), i--vi and 99--130.

\bibitem[Tis]{Tischler1970}
D.~Tischler.
\newblock {On fibering certain foliated manifolds over {$S^{1}$}}.
\newblock {\em Topology} {\bf 9}(1970), 153--154.

\bibitem[Wal]{Wall1967}
C.~T.~C. Wall.
\newblock {Poincar\'e complexes. {I}}.
\newblock {\em Ann. of Math. (2)} {\bf 86}(1967), 213--245.

\bibitem[Weg]{Wegner2009}
Christian Wegner.
\newblock {{$L^2$}-invariants of finite aspherical {CW}-complexes}.
\newblock {\em Manuscripta Math.} {\bf 128}(2009), 469--481.

\bibitem[Zar1]{Zaremsky2016}
Matthew C.~B. Zaremsky.
\newblock {H{NN} decompositions of the {L}odha-{M}oore groups, and topological
  applications}.
\newblock {\em J. Topol. Anal.} {\bf 8}(2016), 627--653.

\bibitem[Zar2]{Zaremsky2018}
Matthew C.~B. Zaremsky.
\newblock {On normal subgroups of the braided {T}hompson groups}.
\newblock {\em Groups Geom. Dyn.} {\bf 12}(2018), 65--92.

\end{thebibliography}

\bigskip
\noindent Dawid Kielak \newline \href{mailto:dkielak@math.uni-bielefeld.de}{\texttt{dkielak@math.uni-bielefeld.de}} \newline
Fakult\"at f\"ur Mathematik  \newline
Universit\"at Bielefeld \newline
Postfach 100131  \newline
D-33501 Bielefeld \newline
Germany \newline

\end{document}